\documentclass[11pt,reqno]{amsart}
\usepackage{amssymb,amsmath,amsfonts,amsthm,enumerate,stmaryrd,tensor,mathtools,dsfont,upgreek,bbm,mathrsfs,nameref,microtype,bm}
\usepackage[dvipsnames]{xcolor}
\usepackage{tikz}
\usetikzlibrary{arrows}
\usepackage[left=0.75 in, right=0.75 in,top=0.75 in, bottom=0.75 in]{geometry}

\usepackage{hyperref,cleveref}

\numberwithin{equation}{section}
\newcommand{\de}{\delta}

\newcommand{\ph}{\varphi}

\newcommand{\st}{~\mid~}

\newcommand{\from}{:}

\renewcommand{\tilde}[1]{\widetilde{#1}}
\renewcommand{\bar}[1]{\overline{#1}}
\renewcommand{\phi}{\varphi}

\newcommand{\innerp}[2]{\left\langle{#1},{#2}\right\rangle}

\DeclarePairedDelimiter\ang{\langle}{\rangle}
\DeclarePairedDelimiter\abs{\lvert}{\rvert}
\DeclarePairedDelimiter\norm{\lVert}{\rVert}

\newtheorem{lemma}{Lemma}[section]
\newtheorem{theorem}[lemma]{Theorem}

\newtheorem{definition}[lemma]{Definition}

\newtheorem{remark}[lemma]{Remark}

\newtheorem{prop}[lemma]{Proposition}

\newcommand{\bfx}{\mathbf{x}}

\newcommand{\scrS}{\mathscr{S}}

\newcommand{\R}{\mathbb{R}}
\newcommand{\C}{\mathbb{C}}
\newcommand{\N}{\mathbb{N}}
\newcommand{\Z}{\mathbb{Z}}
\newcommand{\F}{\mathbb{F}}

\newcommand{\x}{\bfx}

\providecommand{\abs}[1]{\left\vert#1\right\vert}
\providecommand{\norm}[1]{\left\Vert#1\right\Vert}

\providecommand{\br}[1]{\langle #1 \rangle}

\def\dt{\partial_t}

\def\ep{\varepsilon}
\def\vchi{\text{\large{$\chi$}}}
\def\ls{\lesssim}
\def\p{\partial}

\def\XXint#1#2#3{{\setbox0=\hbox{$#1{#2#3}{\int}$ }
		\vcenter{\hbox{$#2#3$ }}\kern-.6\wd0}}

\DeclareMathOperator{\diverge}{div}
\DeclareMathOperator{\supp}{supp}

\newcommand{\fr}{\frac}
\renewcommand{\d}{{\rm d}}

\newcommand{\pr}[1]{\left( #1 \right)}
\renewcommand{\br}[1]{\left[ #1 \right]}

\renewcommand{\sc}{^{{c}}}

\newcommand{\bimp}{\Leftrightarrow}

\newcommand{\sm}{\setminus}

\newcommand{\sseq}{\subseteq}

\renewcommand{\S}{\mathscr{S}}
\newcommand{\stm}{\; \middle| \;}
\newcommand{\loc}{{\mathrm{loc}}}

\usetikzlibrary{decorations.markings}

\def\MarkLt{4pt}
\def\MarkSep{2pt}

\tikzset{
	TwoMarks/.style={
		postaction={decorate,
			decoration={
				markings,
				mark=at position #1 with
				{
					\begin{scope}[xslant=0.2]
						\draw[line width=\MarkSep,white,-] (0pt,-\MarkLt) -- (0pt,\MarkLt) ;
						\draw[-] (-0.5*\MarkSep,-\MarkLt) -- (-0.5*\MarkSep,\MarkLt) ;
						\draw[-] (0.5*\MarkSep,-\MarkLt) -- (0.5*\MarkSep,\MarkLt) ;
					\end{scope}
				}
			}
		}
	},
	TwoMarks/.default={0.5},
}

\author{Subhasish Mukherjee}
\address{
	Department of Mathematical Sciences\\
	Carnegie Mellon University\\
	Pittsburgh, PA 15213, USA
}
\email[S. Mukherjee]{subhasish@cmu.edu}
\author{Ian Tice}
\address{
	Department of Mathematical Sciences\\
	Carnegie Mellon University\\
	Pittsburgh, PA 15213, USA
}
\email[I. Tice]{iantice@andrew.cmu.edu}
\thanks{I. Tice was supported by an NSF Grant (DMS \#2204912). }
\title[On a scale of anisotropic Sobolev spaces ]{
	On a scale of anisotropic Sobolev spaces 
}
\subjclass[2020]{Primary 46E35, 46J15, 35C07; Secondary 46F05, 46N20}

\keywords{Anisotropic Sobolev spaces, Banach algebras, traveling waves}
\begin{document}
	\begin{abstract}
		We introduce a scale of anisotropic Sobolev spaces defined through a three-parameter family of Fourier multipliers and study their functional analytic properties.  These spaces arise naturally in PDE when studying traveling wave solutions, and we give some simple applications of the spaces in this direction.
	\end{abstract}
	
	% _+__+_ -_+__+_ -_+__+_ -_+__+_ -_+__+_ -_+__+_ -_+__+_ -_+__+_ -_+__+_ -_+__+_ -_+__+_ -_+__+_ -_+__+_ -
	\maketitle
	%\tableofcontents
	% _+__+_ -_+__+_ -_+__+_ -_+__+_ -_+__+_ -_+__+_ -_+__+_ -_+__+_ -_+__+_ -_+__+_ -_+__+_ -_+__+_ -_+__+_ -

	% _+__+_ -_+__+_ -_+__+_ -_+__+_ -_+__+_ -_+__+_ -_+__+_ -_+__+_ -_+__+_ -_+__+_ -_+__+_ -_+__+_ -_+__+_ -
	% \maketitle
	% _+__+_ -_+__+_ -_+__+_ -_+__+_ -_+__+_ -_+__+_ -_+__+_ -_+__+_ -_+__+_ -_+__+_ -_+__+_ -_+__+_ -_+__+_ -
	% \pagebreak

	% _+__+_ -_+__+_ -_+__+_ -_+__+_ -_+__+_ -_+__+_ -_+__+_ -_+__+_ -_+__+_ -_+__+_ -_+__+_ -_+__+_ -_+__+_ -
	\section{Introduction}
	% _+__+_ -_+__+_ -_+__+_ -_+__+_ -_+__+_ -_+__+_ -_+__+_ -_+__+_ -_+__+_ -_+__+_ -_+__+_ -_+__+_ -_+__+_ -

	\subsection{Setup and background}

	Consider the problem of finding a solution $v : \R^d \times [0,\infty) \to \F \in \{\mathbb{R},\mathbb{C}\}$ to the equation $\dt v +  (-\Delta)^{\delta/2} v = F$ for some $1 < \delta \in \R$.  When $\delta =2$ this is the standard heat equation.  Let us further assume that $F : \R^d \to \F$ is in traveling wave form, namely $F(x,t) = f(x-\gamma t e_1)$ for some $f : \R^d \to \F$ and traveling wave speed $\gamma \in \R \backslash \{0\}$.  If we make the traveling wave ansatz $v(x,t) = u(x-\gamma t e_1)$, then we reduce to the PDE $-\gamma \partial_1 u + (-\Delta)^{\delta/2} u = f$ in $\R^d$, which rewrites on the Fourier side as 
	\begin{equation}
		[-2\pi i \gamma \xi_1 + (2\pi \abs{\xi})^{\delta} ] \hat{u}(\xi) = \hat{f}(\xi).
	\end{equation}
	Clearly, this determines $\hat{u}$ in terms of $\hat{f}$, and if we assume that $f \in H^s(\R^d;\F)$, then we have the estimate 
	\begin{equation}\label{intro_calc}
		\int_{\R^d} (\abs{\xi_1}^2 + \abs{\xi}^{2\delta}) \ang{\xi}^{2s} \abs{\hat{u}(\xi)}^2 \d\xi \asymp \int_{\R^d} \ang{\xi}^{2s} \abs{\hat{f}(\xi)}^2 \d\xi = \norm{f}_{H^s}^2.
	\end{equation}
	One can show (and we will do so later) that the space defined by the square-norm on the left is complete and consists of locally integrable functions if and only if $d > 1+\delta$.  Thus, in small dimension it is natural to seek a refinement of this estimate (which requires more information on $f$, of course) that overcomes this issue and leads to an isomorphism of Banach spaces for the operator $-\gamma \partial_1 + (-\Delta)^{\delta/2}$.

	To this end, we write $\S(\R^d;\F)$ for the Schwartz space of $\F-$valued functions and $\S'(\R^d;\F)$ for the corresponding space of $\F-$valued tempered distributions.  Given the parameters $s, r, \delta \in \R$ we define the measurable function $\omega_{s, r, \delta} \from \R^d \to [0,\infty)$ via
	\begin{equation}
		\omega_{s, r, \delta}(\xi) = \frac{\abs*{\xi_1}^2 + \abs{\xi}^{2\delta}}{\abs{\xi}^{2r}} \vchi_{B(0, 1)}(\xi) 
		+ \ang{\xi}^{2s} \vchi_{B(0, 1)^c}(\xi),
	\end{equation}
	where $\ang{\xi} = \sqrt{1 +\abs{\xi}^2}$ is the usual bracket notation.  We then define the Sobolev-type space 
	\begin{equation}
		X^s_{r, \delta}(\R^d;\F) = \left\{f \in \S'(\R^d;\F) \stm \hat{f} \in L^1_{\loc}(\R^d;\C) \text{ and } \norm{f}_{X^s_{r, \delta}} < \infty \right\},
	\end{equation}
	where $\hat{\cdot}$ denotes the Fourier transform, and the norm is defined by
	\begin{equation}
		\norm{f}_{X^s_{r, \delta}}^2 = \int_{\R^d} \omega_{s, r, \de}(\xi) \abs*{ \hat{f}(\xi)}^2 \d\xi = \int_{B(0, 1)} \frac{\xi_1^2 + \abs{\xi}^{2\delta}}{\abs{\xi}^{2r}} \abs*{ \hat f(\xi) }^2 \d\xi + \int_{B(0, 1)\sc} \langle{\xi}\rangle^{2s} \abs*{ \hat f(\xi) }^2 \d\xi.\label{asympR}
	\end{equation}
	The norm is clearly derived from the associated inner-product
	\begin{equation}
		\innerp{f}{g}_{X^s_{r, \de}} = \int_{\R^d} \omega_{s, r, \de}(\xi) \hat{f}(\xi) \bar{\hat{g}(\xi)} \,d\xi.
	\end{equation}
	Note that since $\omega_{s,r,\delta}$ is even, this inner-product takes values in $\R$ when $\F = \R$.  We further note that with this notation established, the left side of \eqref{intro_calc} is equivalent to $\norm{u}_{X^{s+\delta}_{0,\delta}}^2$.
	
	Although we have motivated the introduction of $X^s_{r, \delta}(\R^d;\F)$ with a simple linear pseudodifferential equation above, similar issues arose in recent work of the second author and collaborators on the construction of traveling wave solutions to the free boundary Navier-Stokes \cite{leoni_and_tice,stevenson_and_tice,koganemaru_and_tice}  and Muskat systems \cite{nguyen_tice}.  In these instances, the space  $X^s_{1, 2}(\R^d;\F)$ played an essential role in the construction of solutions, and we expect the new more general scale to be useful in other PDE applications.  In particular, for uses in nonlinear PDE, the question of when $X^s_{r, \delta}(\R^d;\F)$ is an algebra is of central importance.

	\subsection{Anisotropic reduction}
	
	Consider the case $\delta \le 1$.  Then for $\xi \in \R^d$ such that $\abs{\xi} \le 1$ we have that 
	\begin{equation}
		\omega_{s,r,\delta}(\xi) = \frac{\abs{\xi_1}^2 + \abs{\xi}^{2\delta}}{\abs{\xi}^{2r}} \asymp \frac{ \abs{\xi}^{2\delta} }{\abs{\xi}^{2r}} = \abs{\xi}^{2(\delta-r)},
	\end{equation}
	and so $X^s_{r,\delta}(\R^d;\F) = \dot{H}^{(\delta-r,s)}(\R^d;\F)$, where for $\lambda,\rho \in \R$ we define the bihomogeneous Sobolev space
	\begin{equation}
		\dot{H}^{(\lambda,\rho)}(\R^d;\F) = \left\{f \in \S'(\R^d;\F) \stm \hat{f} \in L^1_{\loc}(\R^d;\C) \text{ and } \norm{f}_{ \dot{H}^{(\lambda,\rho)}} < \infty \right\}
	\end{equation}
	with
	\begin{equation}
		\norm{f}_{ \dot{H}^{(\lambda,\rho)}}^2 = \int_{B(0, 1)}\abs{\xi}^{2\lambda}\abs*{ \hat f(\xi) }^2 \d\xi + \int_{B(0, 1)\sc} \abs{\xi}^{2\rho} \abs*{ \hat f(\xi) }^2 \d\xi.
	\end{equation}
	This shows that when $\delta \le 1$ the space $X^s_{r,\delta}(\R^d;\F)$ is actually isotropic, and the pair of parameters $(r,\delta)$ reduce to the single parameter $\delta-r\in \R$.  
	
	Similarly, consider the case $d = 1$. We then note that for $\xi \in \R$ with $\abs \xi < 1$ we have
	\begin{align}
		\omega_{s, r, \delta}(\xi) = \frac{\abs {\xi_1}^2 + \abs \xi^{2\delta}}{\abs \xi^{2r}} = \frac{\abs {\xi}^2 + \abs \xi^{2\delta}}{\abs \xi^{2r}} \asymp \abs \xi^{\min\{2(\delta - r), 2(1 - r)\}}
	\end{align}
	and so again we reduce to  $X^s_{r, \delta}(\R; \F) = \dot H^{\delta - r, s}(\R; \F)$ or $X^s_{r, \delta}(\R; \F) = \dot H^{1 - r, s}(\R; \F)$ depending on whether $\delta \le 1$ or $\delta > 1$.
	
	As such, in this paper we will focus our attention on the more interesting regime $\delta >1$ and $d \ge 2$, in which case the space $X^s_{r,\delta}(\R^d;\F)$  is genuinely anisotropic, as we will see later.

	\subsection{Main results}
	
	Our goal in the present paper is two-fold.  First, we aim to study the functional analytic properties of this generalized scale, including embeddings into classical spaces, completeness, and under which parameter regime this space is an algebra.  Second, we will provide some simple uses of these spaces in constructing traveling wave solutions to some simple PDEs to provide an elementary demonstration of the use of this type of space.
	
	The following theorem summarizes the properties of $X^s_{r,\delta}(\R^d;\F)$ we will prove in Sections \ref{sec_prelim} and \ref{sec_algebra}.  Then in Section \ref{sec_PDE} we will record the PDE applications.

	\begin{theorem}
		Let $s,r,\delta \in \R$ and $d \in \N$ with $\delta >1$ and $d \ge 2$.  Then the following hold.  
		\begin{enumerate}
			\item $X^s_{r,\delta}(\R^d;\F)$ is a Hilbert space if and only if $1 +\delta -2r < d$.  In either case, we have the continuous inclusion $X^s_{r,\delta}(\R^d;\F) \hookrightarrow C^\infty_0(\R^d;\F) + H^s(\R^d;\F),$ where $C^\infty_0(\R^d;\F) = \bigcap_{k \in \N} C^k_0(\R^d;\F)$  is endowed with its standard Fr\'{e}chet topology.
			
			\item $H^s(\R^d;\F) \hookrightarrow X^s_{r,\delta}(\R^d;\F)$ if and only if $r \le 1$.   
			
			\item If  $1 +\delta -2r < d$ and $r \le 1$ then $X^s_{r,\delta}(\R^d;\F)$ is anisotropic in the sense that it is not closed under composition with rotations.  More precisely, there exist $f \in X^s_{r,\delta}(\R^d;\F) \cap C^\infty_0(\R^d;\F)$ such that $f \circ Q \notin X^s_{r,\delta}(\R^d;\F)$ whenever $Q \in O(d)$ satisfies $\abs{Qe_1 \cdot e_1} < 1$.  In particular, the subspace inclusion $H^s(\R^d;\F) \subset X^s_{r,\delta}(\R^d;\F)$ is strict in this parameter regime. 
			
			\item If $d > 1 + \delta - 2r$ and $s > d/2$,  $f \in X^s_{r, \delta}(\R^d; \F)$ and $g \in H^s(\R^d; F)$, then $fg \in H^s(\R^d; \F)$ and there exists a constant $C > 0$ such that $\norm*{fg}_{H^s} \le C \norm{f}_{X^s_{r, \delta}} \norm{g}_{H^s}$ 
			
			\item Suppose  $d > 1 + \delta - 2r$, $r \le 1$, and $s > d/2$.  If $d \ge 3$, then $X^s_{r, \delta}(\R^d; \F)$ is an algebra.  If $d = 2$, then $X^s_{r, \delta}(\R^d;\F)$ is an algebra if and only if $\delta \le 2$.  In particular, in this parameter regime, when $X^s_{r, \delta}(\R^d; \F)$ is an algebra, (4) says that $H^s(\R^d;\F) \subset X^s_{r, \delta}(\R^d; \F)$ is an ideal.

			\item  If $\delta - r \le t \le s$ and $f \in X^s_{r, \delta}(\R^d; \F)$,  then $\pr{-\Delta}^{t/2} f \in H^{s - t}(\R^d ;\F)$ and $\partial_1 f \in \dot{H}^{-r}(\R^d;\F)$  and $\norm*{\pr{-\Delta}^{t/2} f}_{H^{s - t}}  + \norm*{\partial_1 f}_{\dot{H}^{-r}} \lesssim \norm*{f}_{X^s_{r, \delta}}$.
		\end{enumerate}
	\end{theorem}
	
	The regions where $X^s_{r, \delta}$ is an algebra are outlined in Figure \ref{fig1}.

	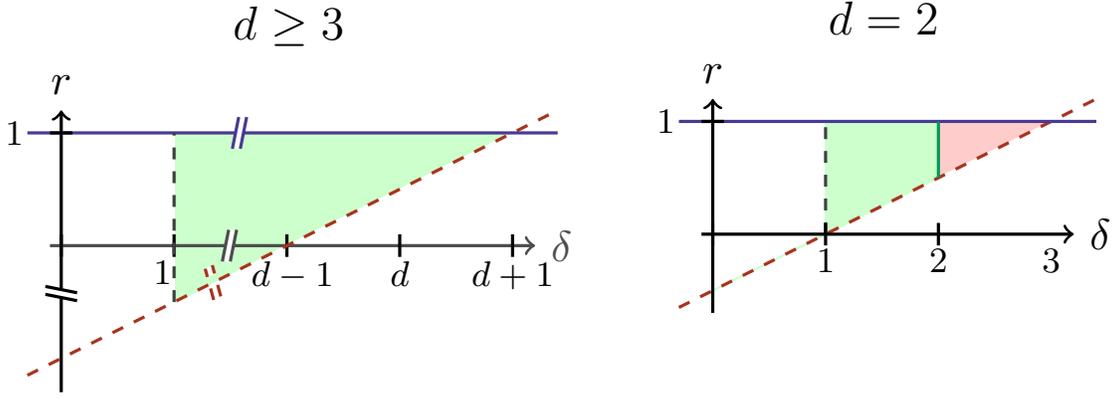
\begin{figure}[!h]
\begin{tabular}{cc}

\scalebox{1.5}{
			\begin{tikzpicture}[domain=0:4]
				%\draw[very thin,color=gray] (-0.1,-1.1) grid (5.1,1.1);
				
				%\draw[thick, ->] (-0.2,0) -- (5.2,0) node[right] {$\delta$};
				\draw[fill=green!20, draw = green!0] plot[smooth, samples=100, domain=2:4] (\x,{(\x-2)/2}) -| (1,-0.525) -- cycle;
				\draw[thick, ->, TwoMarks = 0.35] (0,-1.3) -- (0,1.2) node[above] {$r$};
				\draw[->,thick,black!70,TwoMarks=0.37] (-0.1,0) -- (4.2,0) node[right] {$\delta$}; 
				\foreach \i in {0, 1,...,4} \draw[thick] (\i,-0.1)--(\i,.1);
				\node[below] at (0.9,0) {\footnotesize $1$};
				\node[below] at (2.05,0) {\footnotesize $d-1$};
				\node[below] at (3,0) {\footnotesize $d$};
				\node[below] at (4,0) {\footnotesize $d + 1$};

				\draw[thick, dashed, Mahogany, TwoMarks = 0.35] (-0.3,-1.15) -- (4.4,1.2);
				\draw[thick,BlueViolet,TwoMarks=0.4] (-0.3,1) -- (4.4,1);
				\draw[thick,dashed, darkgray] (1,-0.5) -- (1, 1);
				\draw[thick] (-0.1, 1) -- (0.1, 1);
				\node[left] at (-0.2,1) {\footnotesize $1$};
				%\draw[thick] (-0.1, -1) -- (0.1, -1);
				%\node[right] at (0.1,-1.1) {\footnotesize $\frac{1-d}2$};
				
				%\draw[thick, domain=0:5.1, smooth, variable=\x, purple] plot ({\x}, {1});
				\node[above,font=\large\bfseries] at (current bounding box.north) {$d \ge 3$};
			\end{tikzpicture}
		}
&

		\scalebox{1.5}{
			\begin{tikzpicture}[domain=0:4]
				%\draw[very thin,color=gray] (-0.1,-1.1) grid (5.1,1.1);
				
				%\draw[thick, ->] (-0.2,0) -- (5.2,0) node[right] {$\delta$};
				\draw[fill=green!20, draw = green!20] plot[smooth, samples=100, domain=-0.1:2] (\x,{(\x-1)/2}) -| (1,0) -- cycle;
				\draw[fill=green!20, draw = green!20] plot[smooth, samples=100, domain=1:2] (2,\x/2) -| (1,1/2) -- cycle;
				\draw[fill=red!20, draw = red!0] plot[smooth, samples=100, domain=2:3] (\x,{(\x - 1)/2}) -| (2,0) -- cycle;

				\draw[thick, dashed, Mahogany] (-0.3,-0.65) -- (3.4,1.2);
				\draw[thick,BlueViolet] (-0.3,1) -- (3.4,1);
				\draw[thick,ForestGreen] (2,0.5) -- (2, 1);
				\draw[thick,dashed, darkgray] (1,0) -- (1, 1);
				\draw[thick, ->] (0,-0.7) -- (0,1.2) node[above] {$r$};
                    \draw[thick, white] (0,-1.4) -- (0, -0.7);
				\draw[->,thick,black] (-0.1,0) -- (3.2,0) node[right] {$\delta$}; 
				\foreach \i in {0, 1, 2} \draw[thick] (\i,-0.1)--(\i,.1);
				\node[below] at (1,0) {\footnotesize $1$};
				\node[below] at (2,0) {\footnotesize $2$};
				\node[below] at (3,0) {\footnotesize $3$};
				\draw[thick] (-0.1, 1) -- (0.1, 1);
				
				\node[left] at (-0.2,1) {\footnotesize $1$};
				%\draw[thick, domain=0:5.1, smooth, variable=\x, purple] plot ({\x}, {1});
				\node[above,font=\large\bfseries] at (current bounding box.north) {$d = 2$};
			\end{tikzpicture}
		}

\end{tabular}
		\caption{On the left: For $d \ge 3$, we restrict to the anisotropic region with the grey line marking $\delta > 1$. We show that $X^s_{r, \delta}(\R^d;\F)$ is complete when $d > 1 + \delta - 2r$, the region bounded by the red line. The classical Sobolev space $H^s$ embeds into $X^s_{r, \delta}$ when $r \le 1$, shown by the blue line. In this region, we show that $X^s_{r, \delta}$ is always an algebra.  On the right: We have the same constraints for $d = 2$. However, in this case $X^s_{r, \delta}$ is an algebra if and only if when $\delta \le 2$, marked by the green line and everything to its left. }
		\label{fig1}
	\end{figure}

	% _+__+_ -_+__+_ -_+__+_ -_+__+_ -_+__+_ -_+__+_ -_+__+_ -_+__+_ -_+__+_ -_+__+_ -_+__+_ -_+__+_ -_+__+_ -
	\section{Preliminary estimates}\label{sec_prelim}
	% _+__+_ -_+__+_ -_+__+_ -_+__+_ -_+__+_ -_+__+_ -_+__+_ -_+__+_ -_+__+_ -_+__+_ -_+__+_ -_+__+_ -_+__+_ -
	We begin with some useful estimates and bounds.  To motivate the first we make the following remark.

	\begin{remark}\label{R_ball_remark}
		The unit radius employed in $\omega_{s,r,\de}$ is not essential.  Indeed, it's straightforward to verify that the map
		\begin{equation}
			X^s_{r, \delta}(\R^d;\F) \ni f \mapsto  \left( \int_{B(0, R)} \frac{\xi_1^2 + \abs{\xi}^{2\delta}}{\abs{\xi}^{2r}} \abs*{ \hat f(\xi) }^2 \d\xi + \int_{B(0, R)\sc} \langle{\xi}\rangle^{2s} \abs*{ \hat f(\xi) }^2 \d\xi \right)^{1/2}
		\end{equation}
		yields an equivalent norm for every $R >0$.
	\end{remark}
	
	The next result is a crucial property characterizing integrability of the reciprocal of our multiplier $\omega_{s, r, \delta}$ around the origin.
	
	\begin{lemma}\label{low_freq_est}
		Let  $1\le d \in \N$ and suppose that $r, \delta \in \R$ with $d > {1 + \de - 2r}$ and $\delta > 1$. If $R > 0$, then 
		\begin{align}
			\int_{B(0, R)} \frac{\abs{\xi}^{2r}}{\xi_1^2 + \abs{\xi}^{2\delta}} \,\d\xi < \infty.\label{intsquare}
		\end{align}    
	\end{lemma}
	\begin{proof}
		We first consider the case when $d = 2$, in which case  we compute: 
		\begin{multline}
			\int_{B(0, R)} \frac{\abs{\xi}^{2r}}{\xi_1^2 + \abs{\xi}^{2\delta}} \,\d\xi 
			= \int_0^R \int_0^{2\pi}  \frac{\rho^{2r}}{\rho^2 \cos^2(\theta) + \rho^{2\delta}} \rho \, \d \theta \d \rho \\
			= \int_0^R \rho^{2r - 1} \int_0^{2\pi}  \frac{1}{\cos^2(\theta) + \rho^{2\delta - 2}} \,\d \theta \d \rho = 2\pi \int_0^R \rho^{2r - 1}  \frac{1}{\rho^{\de - 1} \sqrt{1 + \rho^{2\delta - 2}}} \,\d \rho
		\end{multline}
		Since $\delta > 1$, we know $\sqrt{1 + \rho^{2\de - 2}} \asymp 1$ in $B(0, R)$ and so the last integral is finite if and only if $2r - \de > -1 = 1-d$. 
		
		Next suppose $d \geq 3$.  We can write (\ref{intsquare}) in spherical coordinates to find
		\begin{equation}
			I := \int_{B(0, R)} \frac{\abs{\xi}^{2r}}{\xi_1^2 + \abs{\xi}^{2\delta}} \,\d\xi = \int_0^R \int_0^{2\pi} \int_{[0, \pi]^{d - 2}} \frac{{\rho}^{2r}}{(\rho\cos \ph_{d - 2})^2 + \rho^{2\delta}} \rho^{d - 1} (\sin \ph_{d - 2})^{d - 2} g_d(\ph_1, \dots, \ph_{d-3}) \,\d \phi {\rm d}\theta {\rm d}\rho 
		\end{equation}
		where $\d \phi = \prod_{i = 1}^{d - 2} \d \ph_i$ and $g_d(\ph_1, \dots, \ph_{d-3}) = \prod_{i = 1}^{d - 3} \sin^i(\phi_i)$. Integrating over $\ph_1, \dots, \ph_{d-3}$ and changing variables with $u = \sin \phi_{d-2}$ we have  
		\begin{multline}
			I = C(d) \int_0^R \int_0^\pi \frac{\rho^{2r + d - 1}}{\rho^2 \cos^2 \ph_{d-2} + \rho^{2\delta}}(\sin\ph_{d-2})^{d - 2} \;{\rm d}\ph_{d-2} {\rm d}\rho
			= C(d)\int_0^R \int_{-1}^1 \frac{\rho^{2r + d - 1}}{\rho^2 u^2 + \rho^{2\delta}} (1 - u^2)^{\frac{d - 3}{2}}\; {\rm d}u {\rm d}\rho \\
			\le  C(d)\int_0^R \rho^{2r + d - 3} \int_{-1}^1 \frac{1}{u^2 + \rho^{2\delta - 2}}  {\rm d}u {\rm d}\rho 
			=2 C(d) \int_0^R \rho^{2r + d - 3} \frac{1}{\rho^{\delta - 1}} \arctan\left( \frac1{\rho^{\delta - 1}} \right) \; {\rm d}\rho \\
			\asymp  \int_0^R \rho^{2r + d - \delta -2 } \,\d \rho,
		\end{multline}
		where we used that $\arctan(\rho^{1 -\delta}) \asymp 1$ for $\rho < R$ since $\delta > 1$. The latter integral is finite if and only if $d > 1 + \delta - 2r$, giving the desired result.
		%Finally, we consider the case $d=1$ and write 
		%\begin{equation}
		%\int_{B(0, R)} \frac{\abs{\xi}^{2r}}{\xi_1^2 + \abs{\xi}^{2\delta}} \,\d\xi  = \int_{-R}^R  \frac{\xi^{2r-2}}{1+\xi^{2\delta-2}} \,\d\xi = \int_{-R}^R  \frac{\xi^{2r-2\delta}}{1+\xi^{2-2\delta}} \,\d\xi.
		%\end{equation}
		%Since $1 = d > 1 + \delta - 2r \ge 2 - 2r$ we conclude that the integral is finite.
	\end{proof}

	Using the previous lemma, we can now discern when $\norm{\hat f}_{L^1}$ is bounded by $\norm*{f}_{X^s_{r, \delta}}$ and show that functions in $X^s_{r, \delta}$ are sums of a smooth function and something in $H^s$.
	
	\begin{prop}\label{thm_fourier_split_estimates}
		Suppose that $d > 1 + \delta - 2r$, and let $R > 0$.  Then the following hold.
		\begin{enumerate}
			\item There exists a constant $C = C(R, r, d, \delta, s)> 0$ such that if $f \in X^s_{r, \delta}(\R^d;\F)$, then 
			\begin{equation}\label{fourier_split_estimates_0}
				\int_{B(0, R)} \abs*{\hat f(\xi)} \;\d\xi + \pr{\int_{B(0, R)\sc} (1 + \abs \xi^2)^s \abs*{\hat f(\xi)}^2 \; \d\xi}^{1/2} \le C \norm*{f}_{X^s_{r, \delta}}.
			\end{equation}
			Moreover, if $s > d/2$, then $\norm*{ \hat f }_{L^1} \leq C \norm*{f}_{X^s_{r,\delta}}$ for some $C = C(r, d,\delta, s) > 0$.
			
			\item  For $f \in X^s_{r,\delta}(\R^d;\F)$  define the projections $f_{l,R} = (\hat{f} \chi_{B(0,R)})^\vee$ and $f_{h,R} = (\hat{f} \chi_{B(0,R)^c})^\vee$.  Then $f_{l,R}, f_{h,R}\in X^s_{r,\delta}(\R^d;\F)$, $f = f_{l,R} + f_{h,R}$, and we have the bounds $\norm{f_{l,R}}_{X^s_{r,\delta}} \le \norm{f}_{X^s_{r,\delta}}$ and $\norm{f_{h,R}}_{X^s_{r,\delta}} \le \norm{f}_{X^s_{r,\delta}}$.  Moreover, for each $k \in \N$ we have that $f_{l,R} \in C^k_b(\R^d;\F)$ with the estimate $\norm{f_{l,R}}_{C^k_b} \le C(k) \norm{f_{l,R}}_{X^s_{r,\delta}}$, and $f_{h,R} H^s(\R^d;\F)$ with the estimate $\norm{f_{h,R}}_{H^s} \ls \norm{f_{h,R}}_{X^s_{r,\delta}}$.
			
			\item We have the continuous inclusion $X^s_{r,\delta}(\R^d;\F) \hookrightarrow C^\infty_0(\R^d;\F) + H^s(\R^d;\F)$.
		\end{enumerate}
	\end{prop}
	\begin{proof}
		We begin with the proof of the first item.  Clearly, the second term on the left side of \eqref{fourier_split_estimates_0} is bounded by the term on the right, so we only need to consider the first.  We estimate the first term using Cauchy-Schwarz, Lemma \ref{low_freq_est}, and Remark \ref{R_ball_remark}:
		\begin{equation}
			\int_{B(0, R)} \abs*{\hat f(\xi)} \;\d\xi  
			\le \pr{\int_{B(0, R)} \fr{\abs \xi^{2r}}{{ \xi_1^2 + \abs \xi^{2\delta} }} \d\xi}^{1/2} \pr{ \int_{B(0, R)}  \fr{\xi_1^2 + \abs \xi^{2\delta}}{ \abs \xi^{2r} } \abs*{ \hat f(\xi) }^2\d\xi}^{1/2} \leq C \norm*{ f }_{X^s_{r, \delta}}.
		\end{equation}
		Additionally, if $s > d/2$, we can estimate
		\begin{equation}
			\int_{B(0, R)\sc} \abs*{\hat f(\xi)} \; \d\xi \le \pr{ \int_{B(0, R)\sc}  \fr{1}{\pr{ 1 + \abs \xi^2 }^s} \d\xi }^{1/2} \pr{ \int_{B(0, R)\sc}  {\pr{ 1 + \abs \xi^2 }^s} \abs*{ \hat f(\xi) }^2\d\xi }^{1/2}  \leq C\norm*{f}_{X^s_{r, \delta}}.
		\end{equation}
		Combining this with \eqref{fourier_split_estimates_0} gives the second inequality and completes the proof of the first item. The second and third items follows easily from the first and the standard properties of band-limited functions whose Fourier transforms are in $L^1$.
	\end{proof}

	With this lemma in hand, we can now characterize when $X^s_{r, \delta}(\R^d;\F)$ is complete.

	\begin{theorem}\label{thm_complete_char}
		$X^s_{r, \delta}(\R^d;\F)$ is a Hilbert space if and only if $1 + \delta - 2r < d$.
	\end{theorem}
	\begin{proof}
		First suppose that $r > \fr{1 + \delta - d}2$ and that $\{f_k\}_{k \in \N}$ is Cauchy in $X^s_{r, \delta}$.   Write $\mu$ for the measure $\omega_{s, r, \delta}(\xi)\d \xi$ on $\R^d$. Then $\{\hat f_k\}_{k \in \N}$ is Cauchy in $L^2_\mu(\R^d; \F)$, and so there exists some $F \in L^2_\mu(\R^d; \F)$ such that $\hat f_k \to F$ in $L^2_\mu(\R^d; \F)$ as $k \to \infty$.   We now aim to verify that $F \in \scrS' \cap L^1_{\loc}$.   Employing Cauchy-Schwarz and Lemma \ref{low_freq_est}, we have that
		\begin{equation}
			\int_{B(0, 1)} \abs{F} \le \pr{\int_{B(0, 1)} \fr{1}{\omega_{s,r,\delta}}}^{1/2} \pr{\int_{B(0, 1)} \abs*{F}^2{\omega_{s,r,\delta}}}^{1/2} \lesssim \norm*{F}_{L^2_\mu}  < \infty,
		\end{equation}
		or in other words $F \vchi_{B(0, R)} \in L^1$.  Since $\omega_{s,r,\delta}(\xi) = \ang{\xi}^{2s}$ for $\abs{\xi} \ge 1$, it's also clear that $F \vchi_{B(0, R)\sc} \in  \scrS' \cap L^1_\loc$.  Hence, $F = F \vchi_{B(0, R)} + F \vchi_{B(0, R)\sc} \in \in \scrS' \cap L^1_{\loc}$, and so we may define  $f = \check F \in \S'(\R^d;\F)$.  It's then clear that $f \in X^s_{r, \delta}(\R^d;\F)$ and $f_k \to f$ in $X^s_{r, \delta}$, which shows that $X^s_{r, \delta}(\R^d;\F)$ is complete.

		Conversely, suppose $X^s_{r, \delta}(\R^d;\F)$ is complete when $\delta \ge 1$. Consider the norm $\norm{\cdot}_\ast: X^s_{r, \delta}(\R^d; \F) \to [0, \infty)$ defined by $\norm{f}_\ast = \norm{\hat f\vchi_{B(0, 1)}}_{L^1} + \norm{f}_{X^s_{r, \delta}}$, which is well-defined thanks to the fact that $\hat f \in L^1_\loc$ for any $f \in X^s_{r, \delta}$.  Note that $(X^s_{r, \delta}(\R^d; \F); \norm{\cdot}_\ast)$ is also complete.  Then the identity $I:  (X^s_{r, \delta}(\R^d; \F); \norm{\cdot}_\ast)\to (X^s_{r, \delta}(\R^d; \F); \norm{\cdot}_{X^s_{r, \delta}})$ is obviously a continuous surjective linear map, so by the open mapping theorem $\norm{f}_* \lesssim\norm{f}_{X^s_{r, \delta}}$. In particular, we have that $\norm{\hat f\vchi_{B(0, 1)}}_{L^1} \lesssim \norm{f}_{X^s_{r, \delta}}$. 
		
		Now let $0 < \ep < \fr12$ and define the rectangle $R_\ep = [{\ep^\delta}/2, {3\ep^\delta}/{2}] \times \br{\ep/2, {3\ep}/2}^{d - 1}$. Let $F = \vchi_{R _\ep \cup -R_{\ep}}$. Then for $\xi \in R_\ep \cup -R_{\ep}$ we have $\abs*{\xi_1} \asymp \ep^\delta$ and $ \abs\xi \asymp \ep$, and so 
		\begin{equation}
			\omega_{s, r, \delta}(\xi) \asymp \frac{\abs {\xi_1}^2 + \abs \xi^{2\delta}}{\abs \xi^{2r}} \asymp \frac{\ep^{2\delta} + \ep^{2\delta}}{\ep^{2r}} \asymp \ep^{2(\delta - r)} .    
		\end{equation}
		Simple computations show that
		\begin{equation}
			\norm{F}_{L^1} = 2\abs{R_\ep} \asymp \ep^{d - 1 + \delta} \text{ and }\norm*{\check F}_{X^s_{r, \delta}} \asymp \sqrt{\abs{R_\ep} \ep^{2(\delta - r)} } \asymp \ep^{\pr{3\delta+ d - 2r - 1}/{2}}.
		\end{equation}
		Combining the preceding analysis, we have $\ep^{d - 1 + \delta} \asymp \norm{F}_{L^1} \lesssim \norm{\check F}_{X^s_{r, \delta}} \asymp \ep^{\pr{3\delta+ d - 2r - 1}/{2}}$. Sending $\ep \to 0$, this implies $\pr{3\delta+ d - 2r - 1}/{2} < d - 1 + \delta$ giving us the bound $d > 1 + \delta - 2r$ as desired. 
	\end{proof}
	We can also characterize exactly when the Schwartz functions $\S$ and the classical Sobolev space $H^s$ embeds into $X^s_{r, \delta}$. We first prove a useful lemma.
	
	\begin{lemma}\label{schwartz_bound}
		Let  $r, \delta \in \R$ with $\delta \ge 1$. Then
		\begin{equation}
			J = \int_{B(0, 1)} \frac{\abs {\xi_1}^2 + \abs \xi^{2\delta}}{\abs \xi^{2r}} \; \d \xi < \infty
		\end{equation}
		if and only if $r < \frac{2 + d}{2}$, where $B(0, 1)$ is the unit ball in $\R^d$.
	\end{lemma}
	\begin{proof}
		We write $J_1 = \int_{B(0, 1)} \frac{\abs {\xi_1}^2}{\abs \xi^{2r}}\; \d\xi$ and $J_2 = \int_{B(0, 1)} \frac{\abs {\xi}^{2\delta}}{\abs \xi^{2r}}\; \d\xi$, so we have that $J = J_1 + J_2$. We know that $J_2$ is finite if and only if $2(\delta - r) > -d$, so it only remains to analyze $J_1$. We first note by symmetry for any $1 \le k \le d$ we have that
		\begin{equation}
			J_1 = \int_{B(0, 1)} \frac{\abs {\xi_k}^2}{\abs \xi^{2r}}\; \d\xi,
			\text{ and hence } 
			dJ = \sum_{k = 1}^d \int_{B(0, 1)} \frac{\abs {\xi_k}^2}{\abs \xi^2r}\; \d\xi = \int_{B(0, 1)} \frac{\abs {\xi}^2}{\abs \xi^{2r}}\; \d\xi.
		\end{equation}
		The latter is finite if and only if $2 - 2r > - d$. Since $\delta \ge 1$, this is the more restrictive condition, and so we have that $J < \infty$ if and only if $r < \frac{2 + d}{2}$.
	\end{proof}
	
	We now characterize investigate how the Schwartz functions relate to our spaces.
	
	\begin{theorem}\label{thm_schwartz_inclusion}
		For all $s, r \in \R$ and $\delta \ge 1$, we have $\S(\R^d; \F) \cap X^s_{r, \delta}(\R^d; \F)$ is dense in $X^s_{r, \delta}(\R^d; \F)$. Furthermore, we have that $\S(\R^d; \F) \sseq X^s_{r, \delta}(\R^d; \F)$ if and only if $r < \frac {2 + \delta}2$.
	\end{theorem}
	\begin{proof}
		Let $f \in X^s_{r, \delta}(\R^d)$ and $\ep > 0$. For $0 < R_1 < R_2 < \infty$, define the annulus $A(R_1, R_2) = B(0, R_2) \sm B[0, R_1]$. By the monotone convergence theorem, we may find $0 < R_1 < R_2 < \infty$ such that
		\begin{equation}
			\int_{A(R_1, R_2)\sc} \omega_{s, r, \delta}(\xi) \abs*{\hat f(\xi)}^2\; \d \xi < \frac{\ep^2}{4}. 
		\end{equation}
		Pick nonnegative and radially symmetric $\ph \in C_c^\infty(\R^d)$ with $\supp(\ph) \sseq B(0, 1)$ and $\int_{\R^d} \ph = 1$. Then for $0 < \eta < \frac R4$, define the function $F_\eta \in C_c^\infty(\R^d)$ by
		\begin{equation}
			F_\eta(\xi) = \int_{A(R_1, R_2)} \frac{1}{\eta^d} \ph\pr{ \frac{\xi - z}{\eta} } \hat f(z) \; \d z.
		\end{equation}
		An elementary computation shows that $\supp(F_\eta) \sseq A(R_1/2, R_1 + R_2)$ and $\overline{F_\eta(\xi)} = F_\eta(-\xi)$ which then implies $\check F_\eta \in \S(\R^d)$ is real-valued by the same argument as in \cite{leoni_and_tice}. On the annulus $A(R_1/2, R_1 + R_2)$, we know that $\omega_{s, r, \delta}(\xi) \asymp 1$, and so the usual theory of mollification (CITE) supplies us with some $0 < \eta_0 < R_1/2$ such that 
		\begin{equation}
			\int_{A(R_1, R_2)} \omega_{s, r, \delta}(\xi) \abs*{ \hat f(\xi) - F_{\eta_0}(\xi) }^2 \; \d\xi + \int_{A(R_1/2, R_1 + R_2) \sm A(R_1, R_2)} \omega_{s, r, \delta}(\xi) \abs*{F_{\eta_0}(\xi)}^2\; \d\xi < \frac{\ep^2}{8}	
		\end{equation}
		Thus if we define $f_{\eta_0} = \check F_{\eta_0}$, then $f_{\eta_0} \in X^s_{r, \delta}(\R^d) \cap \S(\R^d)$ and $\supp\pr{ \hat f_{\eta_0} } \sseq A(R_1/2, R_1 + R_2)$. Putting everything together, we see
		\begin{multline}
			\norm*{f - f_{\eta_0}}_{X^s_{r, \delta}}^2 = \int_{A(R_1/2, R_1 + R_2)\sc} \omega_{s, r, \delta}(\xi) \abs*{\hat f(\xi)}^2 \; \d\xi + \int_{A(R_1/2, R_1 + R_2)} \omega_{s, r, \delta}(\xi) \abs*{\hat f(\xi) - F_{\eta_0}(\xi)}^2\; \d\xi \\
			< \frac{\ep^2}4 + \int_{A(R_1, R_2)}\omega_{s, r, \delta}(\xi) \abs*{ \hat f(\xi) - F_{\eta_0}(\xi) }^2 \; \d\xi + \int_{A(R_1/2, R_1 + R_2) \sm A(R_1, R_2)} \omega_{s, r, \delta}(\xi) \abs*{\hat f(\xi) - F_{\eta_0}(\xi)}^2\; \d\xi \\
			< \frac{\ep^2}{4} + \int_{A(R_1, R_2)}\omega_{s, r, \delta}(\xi) \abs*{ \hat f(\xi) - F_{\eta_0}(\xi) }^2 \; \d\xi + 2\int_{A(R_1/2, R_1 + R_2) \sm A(R_1, R_2)} \omega_{s, r, \delta}(\xi) \abs*{ F_{\eta_0}(\xi)}^2\; \d\xi\\ + 2\int_{A(R_1, R_2)\sc} \omega_{s, r, \delta}(\xi) \abs*{\hat f(\xi)}^2 \; \d\xi  < \frac{\ep^2}{4} + 2\frac{\ep^2}{8} + 2\frac{\ep^2}{4} < \ep^2.
		\end{multline}
		Since $\ep$ was arbitrary, we have shown that $\S(\R^d; \F) \cap X^s_{r, \delta}(\R^d; \F)$ is dense in $X^s_{r,\delta}(\R^d; \F)$.
		
		For the second assertion, first suppose that $\S(\R^d; \F) \sseq X^s_{r, \delta}(\R^d; \F)$. Then pick some radial $\ph \in \S(\R^d;\R)$ such that $\ph = 1$ on $B(0, 1)$ and $\ph \ge 0$ everywhere. Then $\check \ph \in X^s_{r, \delta}(\R^d; \F)$, so we see that
		\begin{equation}
			\infty > \norm*{\check \ph}_{X^s_{r, \delta}}^2 \ge  \int_{B(0, 1)} \frac{\abs {\xi_1}^2 + \abs \xi^{2\delta}}{\abs \xi^{2r}} \; \d\xi.
		\end{equation}
		By Lemma \ref{schwartz_bound}, we must have $r < \frac{2 + d}2$.\\
		Conversely, suppose that $r < \frac{2 + d}{2}$. Let $\ph \in \S(\R^d; \F)$ and let $\ph_0 = \pr{\hat \ph \vchi_{B(0, 1)}}\;\check{}$\; and $\ph_1 = \pr{\hat \ph \vchi_{B(0, 1)\sc}}\;\check{}$. We then see 
		\begin{align}
			\norm*{\ph_0}_{X^s_{r, \delta}}^2 = \int_{B(0, 1)} \abs*{\hat \ph(\xi)}^2\frac{\abs {\xi_1}^2 + \abs \xi^{2\delta}}{\abs \xi^{2r}} \; \d\xi  < \infty
		\end{align}
		by Lemma \ref{schwartz_bound}, and clearly since $\ph \in \S$ we have
		\begin{align}
			\norm*{\ph_1}_{X^s_{r, \delta}}^2 = \int_{B(0, 1)\sc} \abs*{\hat \ph(\xi)}^2\abs \xi^{2s} \; \d\xi < \infty.
		\end{align}
		We know that $\hat \ph \in \S \sseq L^1_\loc$, thus $\ph \in X^s_{r, \delta}$ as desired.
	\end{proof}
	
	Next we investigate when  standard Sobolev spaces embed into the anisotropic ones.

	\begin{theorem}\label{thm_Hs_inclusion}
		$H^s(\R^d;\F) \hookrightarrow X^s_{r,\delta}(\R^d;\F)$ if and only if $r \le 1$.
	\end{theorem}
	\begin{proof}
		Suppose initially that $r \le 1$. Then for $\abs{\xi} \le 1$ we can use the fact that $\delta > 1$ to bound
		\begin{equation}
			\omega_{s, r, \delta}(\xi) = \fr{\xi_1^2 + \abs \xi^{2\delta}}{\abs \xi^{2r}} \lesssim \abs \xi^{2 - 2r} + \abs \xi^{2\delta - 2r} \lesssim \abs{\xi}^{2-2r} \lesssim 1.
		\end{equation}
		From this we readily deduce the continuous embedding $H^s(\R^d;\F) \hookrightarrow X^s_{r,\delta}(\R^d;\F)$.
		
		Conversely, suppose we have the continuous embedding $H^s(\R^d;\F) \hookrightarrow X^s_{r,\delta}(\R^d;\F)$, and write $C \ge 0$ for the embedding constant.  Restricting to $f \in \S(\R^d;\F)$ such that $\supp(\hat{f}) \subset B(0,1)$, we find that 
		\begin{equation}
			\int_{B(0,1)} \fr{\xi_1^2 + \abs \xi^{2\delta}}{\abs \xi^{2r}}  \abs{\hat{f}(\xi)}^2 \d \xi \le C^2     \int_{B(0,1)} \ang{\xi}^{2s}  \abs{\hat{f}(\xi)}^2 \d \xi,
		\end{equation}
		and since this must hold for all such $f$, we deduce the pointwise bound
		\begin{equation}
			\fr{\xi_1^2 + \abs \xi^{2\delta}}{\abs \xi^{2r}}   \le C^2 \ang{\xi}^{2s} \lesssim 1.
		\end{equation}
		In turn, this implies that $\xi_1^2 \ls \xi_1^{2r}$  for $\abs{\xi_1} < 1$, and hence $r \le 1$.
	\end{proof}

	However, in this parameter regime this space is genuinely bigger than $H^s(\R^d; \F)$, and also anisotropic - in general, if $f \in X^s_{r, \delta}$, it is not true that $f\circ Q \notin X^s_{r, \delta}$ for every nonidentity linear isometry $Q$.
	
	\begin{theorem}\label{thm_anisotropic_verification}
		Suppose that  $r \le 1$, and $ d > \delta - 2r + 1$.  Let $Q \in O(d) = \{M \in \R^{d \times d} \st M^\intercal M = I\}$ be such that $\abs*{Qe_1 \cdot e_1} < 1$. Then there exists $f \in X^s_{r, \delta}(\R^d;\F) \cap C^\infty_0(\R^d;\F) \sm H^s(\R^d;\F)$ such that $f \circ Q \notin X^s_{r, \delta}(\R^d;\F)$.
	\end{theorem}
	\begin{proof}
		Let $Q \in O(d)$ with $\abs*{Qe_1 \cdot e_1} <1$. For $1 \le k \le d$ set $\sigma_k = 1$ if $Qe_1 \cdot e_k \ge 0$ and $\sigma_k = -1$ if $Qe_1 \cdot e_k < 0$.  For $\ep > 0$ set 
		$R_\ep = \sigma_1\br{ {\ep^\delta}/2, {3\ep^\delta}/{2} } \times \prod_{k = 2}^d \sigma_k \br{{\ep}/2, {3\ep}/{2}}$.  By construction, for each $0 < \ep < \frac{2}{3\sqrt  d}$  and $\xi \in R_\ep \cup (- R_\ep) \sseq B(0, 1)$ we have that
		\begin{equation}
			\abs*{Q^\intercal\xi \cdot e_1} = \abs*{\sum_{k = 1}^d \xi_k \pr{e_k \cdot Qe_1 }} = \abs*{\sum_{k = 1}^d \sigma_k\xi_k \abs*{e_k \cdot Qe_1 }} = \sum_{k = 1}^d \abs*{\xi_k} \abs*{e_k \cdot Qe_1} \asymp \ep^\delta\abs*{e_1 \cdot Qe_1} + \sum_{k = 2}^d \abs*{e_k \cdot Qe_1} \ep \asymp \ep,
		\end{equation}
		where the last equivalence follows because $\delta > 1$ and we can find some $j \ge 2$ with $\abs*{e_j \cdot Qe_1}>0$. Furthermore, again because $\delta > 1$, we have that $\abs \xi \asymp \ep$ for $\xi \in R_\ep \cup (- R_\ep)$. We thus readily deduce the equivalences
		\begin{equation}
			\omega_{s, r, \delta}(\xi) = \frac{\abs{\xi_1}^2 + \abs \xi^{2\delta}}{\abs \xi^{2r}} \asymp \frac{\ep^{2\delta} + \ep^{2\delta}}{\ep^{2r}} \asymp \ep^{2\delta - 2r}\text{ and}
		\end{equation}
		\begin{equation}
			\omega_{s, r, \delta}(Q^\intercal \xi) = \frac{\abs*{Q^\intercal \xi \cdot e_1}^2 + \abs*{Q^\intercal \xi}^{2\delta}}{\abs*{Q^\intercal \xi}^{2r}} =\frac{\abs*{Q^\intercal \xi \cdot e_1}^2 + \abs*{\xi}^{2\delta}}{\abs*{ \xi}^{2r}} \asymp \frac{\ep^2 + \ep^{2\delta}}{\ep^2r} \asymp \ep^{2 - 2r}
		\end{equation}
		for $\xi \in R_\ep \cup (- R_\ep)$.  
		
		Set $F_\ep = \vchi_{R_\ep} + \vchi_{-R_\ep}$ and note $F_\ep(-\xi) = F_\ep(\xi) = \overline{F_\ep(\xi)}$.  The previous calculations then show that
		\begin{equation}
			\norm*{F_\ep}_{X^s_{r, \delta}}^2 = \int_{\R^d} \omega_{s, r, \delta}(\xi) \abs*{F_\ep(\xi)}^2 \;\d\xi \asymp \ep^{2\delta - 2r}\ep^{\delta + d - 1} = \ep^{3\delta - 2r + d - 1},
		\end{equation}
		\begin{equation}
			\norm*{F_\ep}_{L^1} = \norm*{F_\ep}_{L^2}^2 = \abs*{R_\ep} + \abs*{-R_\ep} \asymp \ep^{\delta + d - 1},\text{ and}
		\end{equation}
		\begin{equation}
			\norm*{F_\ep \circ Q}_{X^s_{r, \delta}}^2 = \int_{\R^d} \omega_{s, r, \delta}(\xi) \abs*{F_\ep(Q\xi)}^2 \;\d\xi = \int_{\R^d} \omega_{s, r, \delta}(Q^\intercal\xi) \abs*{F_\ep(\xi)}^2 \;\d\xi \asymp \ep^{2 - 2r}\ep^{\delta + d - 1} = \ep^{1 - 2r + \delta + d}.
		\end{equation}

		Let $\alpha = \frac{d + \delta + 1 - 2r}{2}$ and fix $K \in \N$ with $4^K > \frac{3 \sqrt d}{2}$. Set $F = \sum_{k \ge K} 4^{\alpha k} F_{4^{-k}}$ and  note that $\supp(F_{4^{-k}}) \cap \supp(F_{4^{-j}})$ are disjoint for $j, k \ge K$ and $j \ne k$.  Now we compute various norms of $F$.  First, using that $d > \delta - 2r + 1$ we see $\alpha - \delta - d + 1 = \frac{3 - d - \delta - 2r}{2} < 1-\delta  < 0$, and so
		\begin{equation}
			\int_{\R^d} \abs*{F(\xi)} \; \d \xi \asymp \sum_{k = K}^\infty 4^{\alpha k} 4^{-k(\delta + d - 1)} = \sum_{k = K}^\infty 4^{\frac{3d - \delta - 1}2 - r  } < \sum_{k = K}^\infty 4^{k(1 - \delta) } < \infty.
		\end{equation}
		The definition of $\alpha$ requires that $2\alpha - 3\delta - 2r + d - 1 = 2 - 2\delta < 0$; thus
		\begin{equation}
			\int_{\R^d} \omega_{s, r, \delta}(\xi) \abs*{F(\xi)}^2 \; \d \xi \asymp \sum_{k = K}^\infty 4^{2\alpha k} 4^{-k(3\delta - 2r + d - 1)} = \sum_{k = K}^\infty 4^{k(2 - 2\delta)} < \infty.
		\end{equation}
		Because $r \le 1$ we have that $2\alpha - \delta - d + 1 = 2 - 2r \ge 0$; thus
		\begin{equation}
			\norm*{F}_{L^2}^2 = \int_{\R^d} \abs*{F(\xi)}^2\; \d\xi \asymp \sum_{k = K}^\infty 4^{2\alpha k} 4^{-k(\delta + d - 1)} = \sum_{k = K}^\infty 4^{2 - 2r} = \infty.
		\end{equation}
		Finally, $\alpha$ is defined so that $2\alpha - 1 + 2r - d - \delta = 0$, and so
		\begin{equation}
			\norm*{F \circ Q}_{X^s_{r, \delta}}^2 = \int_{\R^d}\omega_{s, r, \delta}(\xi)  \abs*{F(Q\xi)}^2\; \d\xi \asymp \sum_{k = K}^\infty 4^{2\alpha k} 4^{-k(1 - 2r + \delta + d)} = \sum_{k = K}^\infty 1 = \infty.
		\end{equation}
		With this, set $f = \check F$. The previous calculations imply that $f \in X^s_{r, \delta}(\R^d)$ but $f \notin L^2(\R^d) \sseq H^s(\R^d)$ and $f \circ Q \notin X^s_{r, \delta}(\R^d)$. $F$ is compactly supported and in $L^1$, so we conclude that $f \in C^\infty_0(\R^d)$.
	\end{proof}
	
	We close out this section with identifying when derivatives of $X^s_{r, \delta}$ functions lie in classical Sobolev spaces. We give some examples of the uses of these estimates in the last section.
	
	\begin{prop}\label{derivative_estimates}
		Suppose that $1 - r \le s$, $\delta - r \le \tau$, and $\sigma \le s$ and $f \in X^s_{r, \delta}(\R^d; \F)$. Suppose $\ph: \R^d \to [0, \infty)$ satisfies $\ph(\xi) \asymp \abs \xi^\tau$ for $\abs \xi \le 1$ and $\ph(\xi) \asymp \abs \xi^\sigma$ for $\abs \xi \ge 1$. We then have $\ph\pr{\sqrt{-\Delta}} f \in H^{s - \tau}(\R^d ;\F)$ and $\partial_1 f \in \dot{H}^{-r}(\R^d;\F)$  and 
		\begin{equation}
			\norm*{\ph\pr{\sqrt{-\Delta}} f}_{H^{s - \tau}}  + \norm*{\partial_1 f}_{\dot{H}^{-r}} \lesssim \norm*{f}_{X^s_{r, \delta}}.
		\end{equation}
		In particular, when $\delta - r \le \tau \le s$, we have that $\norm*{\pr{-\Delta}^{\tau/2} f}_{H^{s - \tau}} + \norm*{\p_1 f}_{\dot H^{-r}} \lesssim \norm*{f}_{X^s_{r, \delta}}$.
	\end{prop}
	\begin{proof}
		We note since $\delta - r \le \tau$ that $\abs \xi^{2\tau} \leq \abs \xi^{2(\delta - r)} \le \omega_{s, r, \delta}(\xi)$ for $\abs \xi \leq 1$. Thus we have $(1 + \abs \xi^2)^{s - \sigma} \abs \xi^{2\tau}\lesssim \omega_{s, r, \delta}(\xi)$ for $\abs \xi \leq 1$. Clearly we have $(1 + \abs \xi^2)^{s - \sigma} \abs \xi^{2\sigma} \lesssim \abs \xi^{2s} \asymp \omega_{r, s, \delta}(\xi)$ for $\abs \xi \ge 1$, thus we see
		\begin{multline}
			\norm*{\ph\pr{\sqrt{-\Delta}} f}_{H^{s - t}}^2 \asymp \int_{B(0, 1)} (1 + \abs\xi^2)^{s - \sigma} \abs \xi^{2\tau} \abs*{\hat f(\xi)}^2 \; \d\xi + \int_{B(0, 1)\sc} (1 + \abs\xi^2)^{s - \sigma} \abs \xi^{2\sigma} \abs*{\hat f(\xi)}^2 \; \d\xi\\ \lesssim \int \omega_{s, r, \delta}(\xi) \abs*{\hat f(\xi)}^2 \; \d\xi =  \norm*{f}_{X^s_{r, \delta}}^2.
		\end{multline}
		The inclusion and estimate for $\partial_1 f$ follow similarly after we observe that $1 -r  \le s$ implies that $\frac{\abs{\xi_1}^2}{\abs{\xi}^{2r}} \le \abs{\xi}^{2s}$ for $\abs{\xi} \ge 1$. The final claim follows by setting $\ph(\xi) = \abs \xi^\tau$.
	\end{proof}
	
	\section{When is $X^s_{r, \delta}$ an algebra?}\label{sec_algebra}

	We now proceed to our goal of characterizing when $X^s_{r, \delta}$ is an algebra. Our approach  is modeled on the Littlewood-Paley techniques used in \cite{pausader_etal}. We will look to leverage that $H^s$ is an algebra when $s > d/2$. The first step shows that for $f \in X^s_{r, \delta}$ and $g \in H^s$, we know that $fg \in H^s$.
	
	\begin{theorem}\label{ideal}
		Assume that $d > 1 + \delta - 2r$  and $s > d/2$.  Then the following hold.
		\begin{enumerate}
			\item There exists a constant $C > 0$ such that if $f \in X^s_{r, \delta}(\R^d;\F)$ and $g \in H^s(\R^d;\F)$, then $fg \in H^s(\R^d;\F)$ and
			\begin{equation}
				\norm*{fg}_{H^s} \le C \norm*{f}_{X^s_{r, \delta}} \norm*{g}_{H^s}.
			\end{equation}
			\item For $1 \le k \in \N$ the map 
			\begin{equation}
				H^s(\R^d;\F) \times \prod_{j=1}^k X^{s}_{r,\delta}(\R^d;\F) \ni (g,f_1,\dotsc,f_k) \mapsto g \prod_{j=1}^k f_j \in H^s(\R^d;\F)
			\end{equation}
			is a bounded $(k+1)-$linear map.
		\end{enumerate}
	\end{theorem}
	\begin{proof}
		
		The second item follows easily from the first, so we only prove the first.  Fix $f \in X^s_{r,\delta}(\R^d;\F)$ and write $f = f_l + f_h = f_{l,1} + f_{h,1}$ as in Theorem \ref{thm_fourier_split_estimates} with $R=1$.  Since $f_h \in H^s(\R^d;\F)$ and $s > d/2$, we have that $f_h g \in H^s(\R^d;\F)$ and $\norm{f_h g}_{H^s} \ls \norm{f_h}_{H^s} \norm{g}_{H^s} \ls \norm{f}_{X^s_{r,\delta}}\norm{g}_{H^s}$.
		
		On the other hand, since $f_l \in C^k_b(\R^d;\F)$ for every $k \in \N$, we may readily deduce that multiplication by $f_l$ defines a bounded linear map from $H^k(\R^d;\F)$ to itself for every $k \in \N$, and $\norm{f_l \varphi}_{H^k} \ls \norm{f_l}_{C^k_b} \norm{\varphi}_{H^k} \ls \norm{f}_{X^{s}_{r,\delta}}\norm{\varphi}_{H^k}$ for every $\varphi \in H^k(\R^d;\F)$. 		Interpolating, we conclude that multiplication by $f_l$ defines a bounded linear map from $H^t(\R^d;\F)$ to itself for every $0 \le t \in \R$ and $\norm{f_l \varphi}_{H^t} \le C(t)   \norm{f}_{X^{s}_{r,\delta}} \norm{\varphi}_{H^k}$ for every $\varphi \in H^t(\R^d;\F)$, where $C(t)\ge 0$ is a constant that depends on $t$ but is independent of $f$ or $\varphi$.   Picking $t=s$ then shows that $f_l g \in H^s(\R^d;\F)$ and $\norm{f_l g}_{H^s} \ls \norm{f}_{X^s_{r,\delta}} \norm{g}_{H^s}$.
	\end{proof}
	
	We now mark some notation for convenience.
	
	\begin{definition}
		We define the measurable function $\mu_{r, \delta}: \R^d \to [0, \infty)$ by
		$
		\mu_{r, \delta}(\xi) = \frac{\abs {\xi_1} + \abs \xi^\delta}{\abs \xi^r},
		$
		which is asymptotically equivalent to $\sqrt{\omega_{s, r, \delta}(\xi)}$ for $\abs \xi \le 1$.
		We then have that $$\norm*{f}_{X^s_{r, \delta}} \asymp \norm*{\pr{\mu_{r, \delta}\vchi_{B(0, 1)} + \langle\cdot\rangle^s \vchi_{B(0, 1)\sc}}\hat f}_{L^2}.$$
		We then define the trilinear functional $I: \pr{L^0(\R^d; [0, \infty])}^3 \to [0, \infty]$ by
		\begin{equation}
			I(F, G, H) = \int_{B(0, 1)^2} \frac{\mu_{r, \delta}(\xi + \eta)}{\mu_{r, \delta}(\xi) \mu_{r, \delta}(\eta)} F(\xi) G(\eta) H(\xi + \eta) \; \d\xi\; \d\eta,
		\end{equation}
		where $L^0(\R^d; [0, \infty])$ denotes the nonnegative measurable functions on $\R^d$.
	\end{definition}
	
	In fact, $I$ induces a bounded trilinear functional over $\pr{L^2(\R^d; \F)}^3$ as long as $I$ is bounded when restricted to the subset  $\pr{L^2(\R^d; [0, \infty])}^3$.
	
	\begin{lemma}\label{pos}
		Suppose there exists a constant $C > 0$ such that
		\begin{equation}
			I(F, G, H) \le C \norm*{F}_{L^2} \norm*{G}_{L^2}\norm*{H}_{L^2}
		\end{equation}
		for all $F, G, H \in L^2(\R^d; [0, \infty])$. Then $I$ induces a bounded trilinear map over $\pr{L^2(\R^d; \F)}^3$ into $\F$ satisfying the same formula, and there exists some constant $C' > 0$ such that
		\begin{equation}
			\abs*{I(F, G, H)} \le C' \norm*{F}_{L^2} \norm*{G}_{L^2}\norm*{H}_{L^2}
		\end{equation}
	\end{lemma}
	\begin{proof}
		The proof for this is identical to Lemma 2.7 in \cite{koganemaru_and_tice} .
	\end{proof}
	
	Using this lemma, we can identify a crucial link between $X^s_{r,\delta}$ being an algebra and the boundedness of $I$.

	\begin{prop}\label{algebrafunctional}
		Assume that $d > 1 + \delta - 2r$, $r \le 1$, and $s > d / 2$.   There exists a constant $C > 0$ such that
		\begin{equation}
			\norm*{fg}_{X^s_{r, \delta}} \le C\norm*{f}_{X^s_{r, \delta}}\norm*{g}_{X^s_{r, \delta}} \text{ for all }f, g \in X^s_{r, \delta}(\R^d;\F)
		\end{equation}
		if and only if there exists a constant $C > 0$ such that 
		\begin{equation}
			I(F, G, H) = \int_{B(0, 1)^2} \fr{\mu_{r, \delta}(\xi + \eta)}{\mu_{r, \delta}(\xi)\mu_{r, \delta}(\eta)} F(\xi) G(\eta) H(\xi + \eta)\;\d \xi \d \eta \leq C\norm F_{L^2} \norm G_{L^2}\norm H_{L^2}
		\end{equation} for all $F, G, H \in  L^2(\R^d;[0, \infty])$.  
	\end{prop}
	\begin{proof}
		First, suppose that $I$ is bounded and let $f, g \in X^s_{r, \delta}(\R^d;\F)$. We then set $f_0 = \pr{ \vchi_{B(0, 1)} \hat f }\check{} \in C^\infty_0(\R^d;\F)$ and $f_1 = \pr{ \vchi_{B(0, 1)\sc} \hat f }\check{} = f - f_0 \in H^s(\R^d;\F)$.  We define $g_0$ and $g_1$ similarly for $g$.  We then have 
		\begin{equation}
			fg = f_0 g_0 + f_0 g_1 + f_1 g_0 + f_1 g_1.
		\end{equation}
		By  Theorem \ref{thm_Hs_inclusion}, Proposition \ref{ideal},   and the fact that $H^s(\R^d;\F)$ is an algebra, whenever $i + j \ge 1$ we have that $f_i g_j$ is supported in $B(0, 1)\sc$, is in $H^s\hookrightarrow X^s_{r,\delta}$ and
		\begin{equation}
			\norm*{f_i g_j}_{X^s_{r, \delta}} \asymp \norm*{f_i g_j}_{H^s} \lesssim \norm*{f_i}_{X^s_{r, \delta}}\norm*{g_j}_{X^s_{r, \delta}} \leq \norm*{f}_{X^s_{r, \delta}}\norm*{g}_{X^s_{r, \delta}}.
		\end{equation}
		Thus, it only remains to analyze $f_0 g_0$.  Theorem \ref{thm_fourier_split_estimates} shows that  $\hat f_0$ and $\hat g_0$ are integrable and supported in $B(0, 1)$, so Young's inequality implies $\hat f_0 \ast \hat g_0 \in L^1(\R^d;\F)$ and $\supp(\hat f_0 \ast \hat g_0) \sseq B(0, 2)$. 
		Let $\ph \in \mathscr \S(\R^d;\F)$. We employ Tonelli's theorem to calculate
		\begin{multline}
			\int_{\R^d} \mu_{r, \delta} \pr{  \hat f_0 \ast \hat g_0} \ph = \int_{\R^d} \int_{\R^d} \mu_{r, \delta}(\xi + \eta) \hat f_0(\xi) \hat g_0(\eta) \ph(\xi + \eta) \; \d\xi\; \d\eta \\
			= \int_{B(0, 1)^2} \mu_{r, \delta}(\xi + \eta) \hat f_0(\xi) \hat g_0(\eta) \ph(\xi + \eta) \; \d\xi\; \d\eta = I(\mu_{r, \delta} \hat f_0, \mu_{r, \delta} \hat f_0, \ph).
		\end{multline}
		By the assumed boundedness of $I$ we have
		\begin{equation}
			\abs*{ \int_{\R^d} \mu_{r, \delta}(\hat f_0 \ast \hat g_0) \ph } \lesssim \norm*{ \mu_{r, \delta} \hat f_0 }_{L^2} \norm*{ \mu_{r, \delta} \hat g_0 }_{L^2} \norm*{\ph}_{L^2} \lesssim \norm*{ f }_{X^s_{r, \delta}} \norm*{ g}_{X^s_{r, \delta}} \norm*{\ph}_{L^2}.
		\end{equation}
		By the density of $\mathscr S$ in $L^2$, we see that the left hand side extends to define a bounded linear functional on $L^2$ obeying the same estimate, and so the Riesz representation theorem tells us that $\mu_{r, \delta}(\hat f_0 \ast \hat g_0) \in L^2(\R^d)$ and
		\begin{equation}
			\norm*{ f }_{X^s_{r, \delta}} \norm*{ g}_{X^s_{r, \delta}}\gtrsim \norm*{ \mu_{r, \delta}(\hat f_0 \ast \hat g_0) }_{L^2} =  \norm*{ \mu_{r, \delta}\widehat{f_0 g_0}  }_{L^2} \asymp \norm*{f_0 g_0}_{X^s_{r, \delta}}
		\end{equation}
		The last bound followed because $\hat f_0 \ast \hat g_0$ is compactly supported. We thus have the desired result.
		
		Conversely, assume that $X^s_{r, \delta}(\R^d)$ is an algebra. Let $F, G, H \in L^2(\R^d; [0, \infty])$ and note that $I(F, G, H) = I(F \vchi_{B(0, 1)}, G \vchi_{B(0, 1)}, H)$ due to the domain of the integral, and the fact that we have $\pr{ \vchi_{B(0, 1)} F/ \mu_{r, \delta} }\check{}$ and $\pr{ \vchi_{B(0, 1)} G/ \mu_{r, \delta} }\check{}$ are both in $X^s_{r, \delta}$. Thus by Cauchy-Schwarz and the boundedness of products in $X^s_{r, \delta}$ we have
		\begin{multline}
			I(F, G, H) = \int_{\R^d} \mu_{r, \delta} \pr{ \pr{ \vchi_{B(0, 1)} F/ \mu_{r, \delta} } \ast \pr{ \vchi_{B(0, 1)} G/ \mu_{r, \delta} } } H 
			\\\le\; \norm*{ \mu_{r, \delta} \pr{ \pr{ \vchi_{B(0, 1)} F/ \mu_{r, \delta} } \ast \pr{ \vchi_{B(0, 1)} G/ \mu_{r, \delta} } } }_{L^2} \norm*{H}_{L^2} 
			= \norm*{ \pr{ \vchi_{B(0, 1)} F/ \mu_{r, \delta} }\check{} \pr{ \vchi_{B(0, 1)} G/ \mu_{r, \delta} }\check{} \;}_{X^s_{r, \delta}} \norm*{H}_{L^2} \\
			\lesssim\; \norm*{ \pr{ \vchi_{B(0, 1)} F/ \mu_{r, \delta} }\check{}\;}_{X^s_{r, \delta}}\norm*{ \pr{ \vchi_{B(0, 1)} G/ \mu_{r, \delta} }\check{} \;}_{X^s_{r, \delta}} \norm*{H}_{L^2} \lesssim \norm*{F}_{L^2} \norm*{G}_{L^2} \norm*{H}_{L^2}
		\end{multline}
		By Lemma \ref{pos}, we then have the desired result.
	\end{proof}
	
	Thus to prove that $X^s_{r, \delta}$ is a Banach algebra, we wish to analyze the boundedness of the functional $I$. To do this, we first split the domain $B(0, 1)^2$ into two sets to get control of $I$ in each independently.
	\begin{definition}
		We partition $B(0, 1)^2$ into two sets $E_0$ and $E_1$ as follows:
		\begin{equation}
			E_0 = \{ (\xi, \eta) \in B(0, 1)^2 \st \abs{\xi} + \abs{\eta} \le 3 \abs{\abs \xi - \abs \eta} \}  \text{ and }
			E_1 = \{ (\xi, \eta) \in B(0, 1)^2 \st \abs{\xi} + \abs{\eta} > 3 \abs{\abs \xi - \abs \eta} \} 
		\end{equation}
		From this we write $I$ as a sum of two operators $I_0$ and $I_1$, where for $i\in \{0, 1\}$ we have
		\begin{equation}
			I_i = \int_{E_i} \fr{\mu_{r, \delta}(\xi + \eta)}{\mu_{r, \delta}(\xi)\mu_{r, \delta}(\eta)} F(\xi) G(\eta) H(\xi + \eta) \; \d \xi \d \eta
		\end{equation}
	\end{definition}
	
	We now analyze the boundedness of $I_0$ and $I_1$ in turn.
	
	\begin{prop}
		$(\xi, \eta) \in E_1 \bimp (\xi, \eta) \in B(0,1)^2$ and $\fr12 \abs\eta < \abs\xi < 2\abs\eta$. Additionally, $(\xi, \eta) \in E_1$ implies $\abs{\xi + \eta} < 3\abs{\eta}$.
	\end{prop}
	\begin{proof}
		The proof for this can be found in Lemma 2.10 of \cite{koganemaru_and_tice}.
	\end{proof}
	
	\begin{lemma}
		If $(\xi, \eta) \in E_0$, then $\mu_{r, \delta}(\xi + \eta) \lesssim \mu_{r, \delta}(\xi) + \mu_{r, \delta}(\eta)$
	\end{lemma}
	\begin{proof}
		We note since $\delta - r \ge 0$, we have that $\pr{\abs \xi + \abs \eta}^{\delta - r}  \lesssim \abs \xi^{\delta - r} + \abs \eta^{\delta - r}$. Using this, the triangle inequality, and the estimate from the definition of $E_0$, we have
		\begin{multline}\mu_{r, \delta}(\xi+  \eta) = \frac{\abs*{\xi_1 + \eta_1}}{\abs*{\xi + \eta}^r} + \abs*{\xi + \eta}^{\delta - r}  \leq \frac{\abs*{\xi_1} + \abs*{\eta_1}}{\abs*{\abs*\xi - \abs\eta}^r} + \pr{\abs \xi + \abs \eta}^{\delta - r} \\ \lesssim \frac{\abs*{\xi_1} + \abs*{\eta_1}}{\pr{\abs{\xi} + \abs \eta}^r} + \abs \xi^{\delta - r} + \abs \eta^{\delta - r} \leq \mu_{r, \delta}(\xi) + \mu_{r, \delta}(\eta) \end{multline}
	\end{proof}
	
	\begin{prop}\label{E0good}
		$I_0(F, G, H) \lesssim \norm{1/\mu_{r, \delta}}_{L^2} \norm{F}_{L^2} \norm{G}_{L^2} \norm{H}_{L^2}$
	\end{prop}
	\begin{proof}
		The proof is identical to that of Proposition 2.12 in  \cite{koganemaru_and_tice}.
	\end{proof}
	
	\subsection{Splitting $E_1$ further}
	
	We now aim to get more control of $I_1$. To accomplish this, we localize further in $E_1$.
	
	\begin{definition}
		Suppose we have $m, n \in \N$. Then we define
		\begin{equation}
			E_{m, n} = \{(\xi, \eta) \in E_1 \st 2^{-m - 1} < \abs{\xi} \le 2^{-m} \text{ and } 2^{-n +1} < \abs{\xi + \eta} \le 2^{-n + 2}\}.
		\end{equation}
		Similar to before, we also define
		\begin{equation}
			I_{m, n} = \int_{E_{m, n}} \fr{\mu_{r, \delta}(\xi + \eta)}{\mu_{r, \delta}(\xi)\mu_{r, \delta}(\eta)} F(\xi) G(\eta) H(\xi + \eta)\; \d \xi \d \eta
		\end{equation}
		In addition, we define the annulus $A = B[0, 4] \sm B(0, 1/2)$ and set
		\begin{equation}
			F_m = F \vchi_{2^{-m - 1}A}, G_m = G \vchi_{2^{-m - 1}A}, H_n = H \vchi_{2^{-n} A}. 
		\end{equation}
		for $F, G, H \in L^2(\R^d, [0, \infty])$.
	\end{definition}
	
	\begin{lemma}
		We have
		\begin{equation}
			\bigcup_{m = 0}^\infty \bigcup_{n = m}^\infty E_{m, n} = E_1 \text{ and } \sum_{m = 0}^\infty \sum_{n = m}^\infty I_{m, n} = I_1
		\end{equation}
	\end{lemma}
	\begin{proof}
		The proof is identical to that of Lemma 2.14 in \cite{koganemaru_and_tice}.
	\end{proof}
	
	We now come to our first major bound. For large $d$, we will see that in fact we can get control of $I_1$, thus proving that $X^s_{r, \delta}$ is an algebra. For small dimension, we will need to do some further splitting.
	
	\begin{prop}\label{largedlargedelta}
		Let  $r \le 1$.   Let $F, G, H \in L^2(\R^d, [0, \infty])$.  The following hold.
		\begin{enumerate}
			\item If $d \ge 4\delta - 2r - 2$, then 
			\begin{equation}
				I_1(F, G, H) \lesssim \norm{F}_{L^2}\norm{G}_{L^2}\norm{H}_{L^2}
			\end{equation}
			\item If $d < 4\delta - 2r - 2$, then 
			\begin{equation}
				\sum_{m = 0}^\infty \sum_{n > km}^\infty I_{m,n}(F, G, H) \lesssim \norm{F}_{L^2}\norm{G}_{L^2}\norm{H}_{L^2}
			\end{equation}
			for $k = \fr{2(\delta - r)}{1 - r + d/2}$.
		\end{enumerate}
	\end{prop}
	\begin{proof}
		Let $m, n \in \N$ and $n \ge m$. Then we have
		\begin{equation}
			I_{m, n}(F, G, H) \leq I_1(F_m, G_m, H_n) = \int_{E_1} \fr{\mu_{r, \delta}(\xi + \eta)}{\mu_{r, \delta}(\xi)\mu_{r, \delta}(\eta)} F_m(\xi) G_m(\eta) H_n(\xi + \eta)\; \d \xi \d \eta.
		\end{equation}
		The right hand integral vanishes except when $2^{-m - 2} \le \abs \xi$, $\abs \eta \le 2^{-m + 1}$, and $2^{-n - 1} \le \abs{\xi + \eta} \le 2^{-n + 2}$. Thus we have the estimates
		\begin{equation}
			\mu_{r, \delta}(\xi) = \fr{\abs{\xi_1} + \abs{\xi}^\delta}{\abs{\xi}^r} \geq \abs{\xi}^{\delta - r} \gtrsim 2^{-m(\delta - r)},
		\end{equation}
		and similarly $\mu_{r, \delta}(\eta) \gtrsim 2^{-m(\delta - r)}$. Furthermore, we get
		\begin{equation}
			\mu_{r, \delta}(\xi + \eta) = \fr{\abs{\xi_1 + \eta_1} + \abs{\xi + \eta}^\delta}{\abs{\xi + \eta}^r} \le \fr{\abs{\xi + \eta} + \abs{\xi + \eta}^\delta}{\abs{\xi + \eta}^r} \lesssim \fr{2^{-n} + 2^{-n\delta}}{2^{-nr}} \lesssim 2^{-n(1 - r)}.
		\end{equation}
		Thus combining the estimates, we get 
		\begin{equation}
			\fr{\mu_{r, \delta}(\xi + \eta) }{\mu_{r, \delta}(\xi)\mu_{r, \delta}(\eta)} \lesssim \fr{2^{-n(1- r)}}{ 2^{-2m(\delta - r) }} = 2^{2m(\delta - r) - n(1 - r)},
		\end{equation}
		and so we find
		\begin{equation}
			I_{m, n}(F, G, H) \lesssim 2^{2m(\delta - r) - n(1 - r)} \int_{E_1} F_m(\xi) G_m(\eta) H_n(\xi + \eta) \; \d \xi \; \d \eta.
		\end{equation}
		For $\ell \in \Z^d$, let $Q_\ell$ be the closed cube centered at $2^{-n}\ell$ of side length $2^{-n}$ and $\tilde Q_\ell$ denote the closed cube centered at $-2^{-n} \ell$ of side length $9\cdot 2^{-n}$. We note then that 
		\begin{equation}
			\max_{\ell \in \Z^d} \norm*{ \vchi_{Q_\ell} }_{L^2} = \norm*{\vchi_{Q_0}}_{L^2} = 2^{-nd/2}.
		\end{equation}
		Note if $(\xi, \eta) \in E_1$ and $\xi \in Q_\ell$, then
		\begin{equation}
			\abs*{ \eta + 2^{-n}\ell }_\infty \le \abs*{\eta + \xi} + \abs*{-\xi + 2^{-n}\ell}_{\infty} \le 2^{-n + 2} + 2^{-n - 1} = (9/2) 2^{-n},
		\end{equation}
		and so $\eta \in \tilde Q_\ell$. Thus we compute
		\begin{multline}
			\int_{E_1} F_m(\xi) G_m(\eta) H_n(\xi + \eta) \; \d \xi \; \d \eta  
			\le \sum_{ \ell \in \Z^d} \int_{B(0, 1)^2} \pr{ F_m \vchi_{Q_\ell} }(\xi) \pr{G_m \vchi_{\tilde Q_\ell}}(\eta) H_n(\xi + \eta)\;\d\xi\;\d\eta \\
			\le 
			\sum_{\ell \in \Z^d} \int_{B(0, 1)} \pr{F_m \vchi_{Q_\ell}}(\xi) \norm*{ G_m \vchi_{\tilde Q_\ell} }_{L^2} \norm*{H_n}_{L^2} \;\d\xi 
			\le 
			\norm*{H_n}_{L^2} \norm*{\vchi_{Q_0}}_{L^2} \sum_{\ell \in \Z^d} \norm*{ F_n \vchi_{Q_\ell} }_{L^2} \norm*{ G_m \vchi_{\tilde Q_\ell} }_{L^2} \\
			\le 
			2^{-nd/2} \norm*{H_n}_{L^2} \pr{ \int_{\R^d} \abs*{F_m(\xi)}^2 \sum_{\ell \in \Z^d} \vchi_{Q_l}(\xi)\; \d\xi  }^{1/2} \pr{ \int_{\R^d} \abs*{G_m(\eta)}^2 \sum_{\ell \in \Z^d} \vchi_{\tilde Q_l}(\eta)\; \d\eta } \\
			\lesssim 2^{-nd/2} \norm*{F_m}_{L^2} \norm*{G_m}_{L^2} \norm*{H_n}_{L^2}.
		\end{multline}
		Synthesizing these, we obtain 
		\begin{equation}
			I_{m, n}(F, G, H) \lesssim 2^{2m(\delta - r) - n(1 - r + d/2)}.
		\end{equation}
		
		Now we break to cases based on dimension. If $d \ge 4\delta -2r - 2$, then we simply bound
		\begin{multline}
			\sum_{m \ge 0} \sum_{n \ge m} I_{m, n}(F, G, H) \lesssim \sum_{m \ge 0} \sum_{n \ge m} 2^{2m(\delta - r) - n(1 - r + d/2)}  \norm{F_m}_{L^2}\norm{G_m}_{L^2}\norm{H_n}_{L^2} \\
			\le \sum_{m = 0}^\infty 2^{2m(\delta - r)}  \norm{F_m}_{L^2}\norm{G_m}_{L^2}\pr{\sum_{n = m}^\infty 2^{-n(2 - 2r + d)}}^{1/2} \pr{\int \abs{H}^2 \sum_{n = m}^\infty \vchi_{2^{-n} A} }^{1/2} \\
			\lesssim \norm{H}_{L^2} \sum_{m = 0}^\infty 2^{m(2\delta - r - 1 - d/2)} \norm{F_m}_{L^2}\norm{G_m}_{L^2} \lesssim \norm{F}_{L^2}\norm{G}_{L^2}\norm{H}_{L^2}.
		\end{multline}
		On the other hand, if $d < 4\delta -2r - 2$, then we split further: for $k = \fr{2(\delta - r)}{1 - r + d/2}$ an analogous argument shows that
		\begin{equation}
			\sum_{m \ge 0} \sum_{n \ge km} I_{m, n}(F, G, H) \lesssim \norm{H}_{L^2} \sum_{m = 0}^\infty 2^{m(2\delta - 2r - k + kr - kd/2)} \norm{F_m}_{L^2}\norm{G_m}_{L^2} \lesssim \norm{F}_{L^2}\norm{G}_{L^2}\norm{H}_{L^2}.
		\end{equation}
		
	\end{proof}
	
	We now prove our final bound when $d$ is small.
	
	\begin{prop}\label{smalldlargedelta}
		Suppose that $r \le 1$ and $d < 4\delta - 2r - 2$, and if $d =2$ further suppose that $\delta \le 2$.  For $F, G, H \in L^2(\R^d, [0, \infty])$ we have the estimate
		\begin{equation}\label{bigsum}
			\sum_{m \ge 0} \sum_{m \le n \le km} I_{m, n}(F, G, H) \lesssim \norm{F}_{L^2}\norm{G}_{L^2}\norm{H}_{L^2},
		\end{equation}
		where $k = \fr{2(\delta - r)}{1 - r + d/2}$.
	\end{prop}
	
	\begin{proof}
		We define
		\begin{equation}
			R_{p, \pi}(\alpha) = 2^{-\delta m}[-\alpha/2 + p, \alpha/2 + p] \times 2^{-n}\prod_{k = 2}^{d} [-\alpha/2 + \pi_{k - 1}, \alpha/2 + \pi_{k - 1}]
		\end{equation}
		For $\xi \in R_{p, \pi}(1)$ and $\eta \in R_{q, \sigma}(1)$, we get
		\begin{equation}
			2^{-n}(\pi_{k - 1} + \sigma_{k - 1} - 1/2) \le \xi_{k} + \eta_k \le 2^{-n}(\pi_{k - 1} + \sigma_{k - 1} - 1/2)
		\end{equation}
		Combining with the fact that $\abs{\xi + \eta} \le 2^{-n + 2}$, we get
		\begin{equation}
			\abs{\pi_{k - 1} + \sigma_{k - 1}} \le \abs{\pi_{k - 1} + \sigma_{k - 1} - 2^n(\xi_k + \eta_k)} + \abs{2^n(\xi_k + \eta_k)} \le 9/2
		\end{equation}
		In particular, $\eta \in R_{q, -\pi}(9)$. Combined with $\xi \in R_{p, \pi}(1)$, we get $\xi + \eta \in R_{p + q, 0}(10)$. Thus
		\begin{multline}
			I_{m, n}(F, G, H) \\
			\le \sum_{p, q \in \Z} \sum_{\pi, \sigma \in \Z^{d - 1}} \int_{E_1} \fr{\mu_{r, \delta}(\xi + \eta)}{\mu_{r, \delta}(\xi)\mu_{r, \delta}(\eta)} (F_m \vchi_{R_{p, \pi}}(1))(\xi) (G_m \vchi_{R_{q, -\pi}}(9))(\eta) (H_n \vchi_{R_{p + q, 0}}(10))(\xi + \eta) \d \xi \d \eta.
		\end{multline}
		Now if $\xi \in R_{p, \pi}(1) \cap 2^{-m - 1}A$, then $2^{\delta m} \abs{\xi_1} \ge \abs{p} - 1/2$, so
		\begin{equation}
			\mu_{r, \delta}(\xi) = \fr{\abs{\xi_1} + \abs{\xi}^\delta}{\abs{\xi}^r}\gtrsim \fr{2^{-\delta m}\abs p + 2^{-\delta m}}{2^{-m r}} \gtrsim 2^{-m(\delta-r)}\abs{p}
		\end{equation}
		and similarly
		\begin{equation}
			\mu_{r, \delta}(\eta) \gtrsim 2^{-m(\delta -r)}\abs{q}.
		\end{equation}
		Then because $\xi + \eta \in R_{p + q, 0}(10) \cap 2^{-n-1}A$, we get
		\begin{multline}
			\mu(\xi + \eta) = \fr{\abs{\xi_1 + \eta_1} + \abs{\xi + \eta}^\delta}{\abs{\xi + \eta}^r}
			\lesssim 2^{nr}(\abs{\xi_1 + \eta_1} + 2^{-n\delta}) \\ = 2^{nr - m\delta }(2^{m\delta }\abs{\xi_1 + \eta_1} + 2^{\delta(m - n)}) \lesssim 2^{nr - m\delta}(\abs{p} + \abs q)
		\end{multline}
		Putting everything together, we get
		\begin{equation}
			\frac{\mu_{r, \delta}(\xi + \eta)}{\mu_{r, \delta}(\xi)\mu_{r, \delta}(\eta)} \lesssim 2^{nr + m (\delta - 2r)} \pr{\fr 1{\abs p} + \fr 1{\abs q}}
		\end{equation}
		Then for fixed $m, n \in \N$, we see
		\begin{multline}
			2^{-nr - m(\delta - 2r)}I_{m,n}(F, G, H) \\ 
			\le  \sum_{p, q \in \Z} \sum_{\pi \in \Z^{d - 1}} \int_{B(0, 1)} \pr{\fr 1{\abs p} + \fr1 {\abs q}} (F_m \vchi_{R_{p, \pi}(1)})(\xi)\int_{B(0, 1)} (G_m \vchi_{R_{q, -\pi}(9)})(\eta) (H_n \vchi_{R_{p + q, 0}(10)})(\xi + \eta) \;\d \eta \;\d \xi \\
			\le  \sum_{p, q \in \Z} \sum_{\pi\in \Z^{d - 1}} \int_{B(0, 1)} \pr{\fr 1{\abs p} + \fr1 {\abs q}} F_m(\xi)\vchi_{R_{p, \pi}(1)}(\xi) \norm*{G_m \vchi_{R_{q, -\pi}(9)}}_{L^2} \norm*{H_n \vchi_{R_{p + q, 0}(10)}}_{L^2}\; \d \xi \\
			\le \sum_{p, q \in \Z} \sum_{\pi \in \Z^{d - 1}}  \pr{\fr 1{\abs p} + \fr1 {\abs q}} \norm*{F_m \vchi_{R_{p, \pi}(1)}}_{L^2} \norm*{ \vchi_{R_{p, \pi}(1)}}_{L^2} \norm*{G_m \vchi_{R_{q, -\pi}(9)}}_{L^2} \norm*{H_n \vchi_{R_{p + q, 0}(10)}}_{L^2} \\
			= 2^{\pr{-\delta m - (d - 1)n}/2}\sum_{p, q \in \Z} \sum_{\pi \in \Z^{d - 1}}  \pr{\fr 1{\abs p} + \fr1 {\abs q}} \norm*{F_m \vchi_{R_{p, \pi}(1)}}_{L^2} \norm*{G_m \vchi_{R_{q, -\pi}(9)}}_{L^2} \norm*{H_n \vchi_{R_{p + q, 0}(10)}}_{L^2} 
		\end{multline}
		Firstly, we handle the sum over $\pi \in \Z^{d - 1}$. Indeed, we find for each $p, q \in \Z$
		\begin{align}
			\sum_{\pi \in \Z^{d - 1}} \norm*{F_m \vchi_{R_{p, \pi}(1)}}_{L^2} \norm*{G_m \vchi_{R_{q, -\pi}(9)}}_{L^2} &\le \pr{ \int_{\R^d} \abs*{F_m}^2 \sum_{\pi \in \Z^{d - 1}}  \vchi_{R_{p, \pi}(1)}}^{1/2} \pr{ \int_{\R^d} \abs*{G_m}^2 \sum_{\pi \in \Z^{d - 1}}  \vchi_{R_{q, -\pi}(9)}}^{1/2} \\
			&\le \norm*{F_m \vchi_{ \bigcup_{\pi \in \Z^{d - 1}} R_{p, \pi}(1)} }_{L^2} \norm*{G_m \vchi_{ \bigcup_{\pi \in \Z^{d - 1}} R_{q, -\pi}(9)} }_{L^2}
		\end{align}
		Now we consider the sums over $p, q$. First we consider the term containing $\frac{1}{\max\{1,\abs p\}}$. We find
		\begin{align}
			&\sum_{p, q \in \Z} {\fr 1{\abs p} } \norm*{F_m \vchi_{ \bigcup_{\pi \in \Z^{d - 1}} R_{p, \pi}(1)} }_{L^2} \norm*{G_m \vchi_{ \bigcup_{\pi \in \Z^{d - 1}} R_{q, -\pi}(9)} }_{L^2} \norm*{H_n \vchi_{R_{p + q, 0}(10)}}_{L^2}  \\
			\lesssim \; &\norm*{G_m \vchi_{\bigcup_{q \in \Z} \bigcup_{\pi \in \Z^{d - 1}} R_{q, -\pi}(9)} }_{L^2} \sum_{p \in \Z} \frac{1}{\max\{1, \abs p\}} \norm*{F_m \vchi_{ \bigcup_{\pi \in \Z^{d - 1}} R_{p, \pi}(1)} }_{L^2} \norm*{H_n \vchi_{\bigcup_{q \in \Z} R_{p + q, 0}(10)}}_{L^2} \\
			\le\;&\norm*{G_m}_{L^2} \norm*{H_n}_{L^2}\pr{\sum_{p \in \Z} \frac{1}{\max\{1, \abs p^2\}}}^{1/2} \norm*{F_m \vchi_{ \bigcup_{p \in \Z} \bigcup_{\pi \in \Z^{d - 1}} R_{p, \pi}(1)} }_{L^2} \lesssim  \norm*{F_m}_{L^2} \norm*{G_m}_{L^2} \norm*{H_n}_{L^2}
		\end{align}
		Analogously, we can compute the sum containing $\frac{1}{\max\{1, \abs q\}}$ to find
		\begin{equation}
			\sum_{p, q \in \Z} {\fr 1{\abs q} } \norm*{F_m \vchi_{ \bigcup_{\pi \in \Z^{d - 1}} R_{p, \pi}(1)} }_{L^2} \norm*{G_m \vchi_{ \bigcup_{\pi \in \Z^{d - 1}} R_{q, -\pi}(9)} }_{L^2} \norm*{H_n \vchi_{R_{p + q, 0}(10)}}_{L^2} \lesssim  \norm*{F_m}_{L^2} \norm*{G_m}_{L^2} \norm*{H_n}_{L^2}
		\end{equation}
		Combining the previous bounds, we then arrive at the estimate
		\begin{equation}\label{bigbound}
			I_{m, n}(F, G, H) \lesssim 2^{n\pr{r - \fr d2 + \fr 12} + m\pr{\fr \delta 2 - 2r} } \norm{F_m}_{L^2} \norm{G_m}_{L^2} \norm{H_n}_{L^2}.
		\end{equation}
		
		We now split to two cases. In the first assume that $r - \fr d2 + \fr 12 \le 0$, looking at the inner sum in (\ref{bigsum}) we bound using the first term to find
		\begin{multline}
			\sum_{n = m}^{km} I_{m,n}(F, G, H) \lesssim 2^{m\pr{\fr \delta 2 - 2r} } \pr{\sum_{n = m}^{km} 2^{n(2r - d + 1)}}^{1/2} \norm{H}_{L^2} \norm{F_m}_{L^2} \norm{G_m}_{L^2} \\
			\lesssim 2^{m\pr{\pr{\fr \delta2 - 2r} + \pr{r - \fr d2 + \fr 12}}  } \norm{H}_{L^2} \norm{F_m}_{L^2} \norm {G_m}_{L^2}
		\end{multline}
		Now summing over $m$, this converges if  and only if
		$
		\fr{\delta - d + 1}{2} \le r
		$
		which is true by hypothesis. 
		
		For the other case, we have $r - \fr d2 + \fr 12 > 0$. In particular, since we have already restricted $r \le 1$ this implies that $d = 2$. Plugging in $d = 2$ in (\ref{bigbound}), we again analyze the inner sum of (\ref{bigsum}), this time bounding using the last term to see   \begin{multline}
			\sum_{n = m}^{km} I_{m,n}(F, G, H) \lesssim 2^{m\pr{\fr \delta 2 - 2r} } \pr{\sum_{n = m}^{km} 2^{n(2r - 2 + 1)}}^{1/2} \norm{H}_{L^2} \norm{F_m}_{L^2} \norm{G_m}_{L^2} \\
			\lesssim 2^{m\pr{\pr{\fr \delta 2 - 2r} + k(r - 1/2)}} \norm{H}_{L^2} \norm{F_m}_{L^2} \norm{G_m}_{L^2} 
			\lesssim 2^{m\pr{\pr{\fr \delta 2 - 2r} + \fr{2(\delta - r)}{2 - r}(r - 1/2)}} \norm{H}_{L^2} \norm{F_m}_{L^2} \norm{G_m}_{L^2}
		\end{multline}
		Again summing over $m$, this converges when
		\begin{multline}
			\pr{\fr \delta 2 - 2r} + \fr{2(\delta - r)}{2 - r}(r - 1/2) \le 0 
			\bimp (2 - r)\pr{\fr \delta 2 - 2r} + (\delta - r)\left(2r -1\right)  \le 0 \\
			\bimp  \fr \delta 2- 2r - \fr{r\delta}2 + 2r^2 + \fr{d \delta}{4} - rd + 2\delta r - 2r^2 -2\delta + 2r + \delta - r \le 0  \bimp \delta \le 2,
		\end{multline}
		the latter condition of which is assumed to hold by hypothesis.
	\end{proof}
	
	In fact, the assumption that $\delta \le 2$ when $d = 2$ was necessary, as the next proposition shows.
	
	\begin{prop}\label{counterexample}
		Suppose that $d = 2$, $r >0$, and $\delta > 2$. Then there does not exist a constant $C > 0$ such that 
		\begin{equation}
			\int_{B(0, 1)^2} \fr{\mu_{r, \delta}(\xi + \eta)}{\mu_{r, \delta}(\xi)\mu_{r, \delta}(\eta)} F(\xi) G(\eta) H(\xi + \eta) \; \d \xi \; \d \eta \le C \norm*{F}_{L^2}\norm*{G}_{L^2} \norm*{H}_{L^2}
		\end{equation}
		for every $F, G, H \in L^2(\R^2)$.
	\end{prop}
	\begin{proof}
		Define the sets $Q = B_\infty((2^{-m\delta}, 2^{-m}), 2^{-m\delta - 2})$, $Q' = B_\infty((2^{-m\delta}, -2^{-m}), 2^{-m\delta - 2})$, and $P = Q + Q' = B_\infty((2^{-m\delta}, 0), 2^{-m\delta - 1})$. We then note for $\xi \in Q$, we have that $\abs \xi \asymp 2^{-m}$, and so
		\begin{equation}
			\mu_{r, \delta}(\xi) = \fr{\abs*{\xi_1} + \abs \xi^{\delta}}{\abs \xi^r} \lesssim \fr{2^{-m\delta} + 2^{-m\delta}}{2^{-m r}} \lesssim 2^{-m(\delta - r)}.
		\end{equation}
		and similarly $\mu_{r, \delta}(\eta) \lesssim 2^{-m(\delta - r)}$ when $\eta \in Q'$. We also note $\xi + \eta \in P$ when $\xi \in Q, \eta \in Q'$, and so $\abs \xi \asymp \abs \eta \asymp 2^{-m \delta}$ and $\abs {\xi + \eta} \asymp 2^{-m\delta}$. Thus we see
		\begin{equation}
			\mu_{r, \delta}(\xi + \eta)= \fr{\abs*{\xi_1 + \eta_1} + \abs {\xi + \eta}^{\delta}}{\abs {\xi + \eta}^r} \gtrsim \fr{2^{-m \delta} + 2^{-m \delta^2}}{2^{-m\delta r}} \gtrsim 2^{-m \delta ( 1- r)}.
		\end{equation}
		Finally, we note that $\mu(Q) \asymp \mu(Q') \asymp \mu(P) \asymp 2^{-2m\delta}$ and so $\norm*{\vchi_{Q}}_{L^2}  \norm*{\vchi_{Q'}}_{L^2}  \norm*{\vchi_{P}}_{L^2} \asymp 2^{-3m\delta}$. On the other hand, we can compute
		\begin{multline}
			\iint_{B(0, 1)^2} \fr{\mu_{r, \delta}(\xi + \eta)}{\mu_{r, \delta}(\xi)\mu_{r, \delta}(\eta)} \vchi_{Q}(\xi) \vchi_{Q'}(\eta) \vchi_{P}(\xi + \eta) \; \d \xi \; \d \eta = \iint_{Q \times Q'} \fr{\mu_{r, \delta}(\xi + \eta)}{\mu_{r, \delta}(\xi)\mu_{r, \delta}(\eta)}  \; \d \xi \; \d \eta \\
			\gtrsim 2^{-m(\delta(1 - r) - 2(\delta - r))} \mu(Q) \mu(Q')  
			\gtrsim 2^{-m(4\delta + \delta(1 - r) - 2(\delta - r))}.
		\end{multline}
		In particular, if the linear functional was bounded, we could find some $C > 0$ such that
		\begin{equation}
			2^{-m(4\delta + \delta(1 - r) - 2(\delta - r) - 3\delta)} \le C
		\end{equation}
		for all $m \in \N$. However,  since $\delta > 2$ and $r >0$ we know that 
		\begin{equation}
			4\delta + \delta(1 - r) - 2(\delta - r) - 3\delta = r(2 - \delta) < 0,
		\end{equation}
		and so making $m$ arbitrarily large, we contradict the inequality.
	\end{proof}
	
	We can now state our main result.
	
	\begin{theorem}\label{thm_algebra_result}
		Suppose  $d > 1 + \delta - 2r$, $r \le 1$, $s > d/2$.  If $d \ge 3$, then $X^s_{r, \delta}(\R^d; \F)$ is an algebra. If $d = 2$, then $X^s_{r, \delta}$ is an algebra if and only if $\delta \le 2$.
	\end{theorem}
	\begin{proof}
		The $d \ge 3$ and $d = 2, \delta \le 2$ cases follows by combining Propositions \ref{algebrafunctional}, \ref{E0good}, \ref{largedlargedelta}, \ref{smalldlargedelta}. For $d = 2$, $\delta > 2$, we note that we must have $r > 1/2 > 0$ due to our other restrictions, and so Proposition \ref{counterexample} shows that $X^s_{r, \delta}$ is not an algebra in this case.
	\end{proof}

	% _+__+_ -_+__+_ -_+__+_ -_+__+_ -_+__+_ -_+__+_ -_+__+_ -_+__+_ -_+__+_ -_+__+_ -_+__+_ -_+__+_ -_+__+_ -
	\section{PDE applications}\label{sec_PDE}
	% _+__+_ -_+__+_ -_+__+_ -_+__+_ -_+__+_ -_+__+_ -_+__+_ -_+__+_ -_+__+_ -_+__+_ -_+__+_ -_+__+_ -_+__+_ -
	
	In this section we give a couple short and simple applications of the $X^s_{r,\delta}(\R^d;\F)$ spaces in constructing traveling wave solutions to PDEs.  	Recall the motivating pseudodifferential equation equation from the introduction: $\dt v + (-\Delta)^{\delta/2}  = F$, which reduces to $-\gamma \partial_1 u + (-\Delta)^{\delta/2} u = f$ after the traveling wave reformulation.  We assume here that $\gamma \neq 0$, which corresponds to actual traveling wave solutions and not stationary solutions.  Here the operator $(-\Delta)^{\delta/2}$ comes from a homogeneous function on the Fourier side, namely $\R^d \ni \xi \mapsto (2\pi \abs{\xi})^{\delta} \in [0,\infty)$.

	In fact, the spaces $X^s_{r,\delta}(\R^d;\F)$ are designed to be more flexible by handling more general symbols with a manifest ``bihomogeneity,'' meaning possibly different homogeneous behavior for small and large frequencies.  	To describe this, we let $\varphi : \R^d \to [0,\infty)$ be a continuous function such that 
	\begin{equation}\label{bihom_cond}
		\varphi(\xi) \asymp
		\begin{cases}
			C_0 \abs{\xi}^{\delta} &\text{for } \abs{\xi} \ll 1 \\
			C_1 \abs{\xi}^{\sigma} &\text{for } \abs{\xi} \gg 1
		\end{cases}
	\end{equation}
	for $\delta,\sigma \in \R$ satisfying $\delta >1$ and $\sigma \in \R$.  We write $D = \sqrt{-\Delta}$ and $\varphi(D)$ for the pseudodifferential operator acting via $\widehat{\varphi(D) u}(\xi) = \varphi(\xi) \hat{u}(\xi)$.  As two particular examples: (1) the function $\varphi(\xi) = (2\pi \abs{\xi})^{\delta}$ gives $\varphi(D) = (-\Delta)^{\delta/2}$ from the introduction, and satisfies $\delta = \sigma$; (2) the function $\varphi(\xi) = \abs{\xi} \tanh(\abs{\xi} )$ is of this type with $\delta = 2$ and $\sigma = 1$; this particular $\varphi$ arose in the analysis in \cite{nguyen_tice} and is related to the classic gravity-wave dispersion relation.  We can then consider the modification of the previous pseudodifferential equation: $\p_t v + \ph(D) u = F$, which reduces to $-\gamma \p_1 u + \ph(D) u = f$ after the traveling wave reformulation. Our first result establishes solvability of this linear problem.

	\begin{theorem}\label{pde_double_linear}
		Let $s, r, \delta, \sigma \in \R$ satisfy $1 < \delta$, $d > 1 + \delta - 2r$, $r \ge 0$, $\sigma \le s$, and $1 - r \le s$.  Suppose $\ph: \R^d \to [0, \infty)$ is a continuous function satisfying \eqref{bihom_cond}.  Let $\beta,\gamma \in \R \backslash \{0\}$. Then the map $-\gamma \partial_1 + \beta \ph(D) : X^{s+\sigma}_{r,\delta}(\R^d;\F) \to (H^s \cap \dot{H}^{-r})(\R^d;\F)$ is well-defined and induces a bounded linear isomorphism.  In particular, for each $f \in (H^s \cap \dot{H}^{-r})(\R^d;\F)$ there exists a unique $u \in X^{s+\sigma}_{r,\delta}(\R^d;\F)$ solving 
		\begin{equation}\label{pde_double_linear_1}
			-\gamma \partial_1 u + \beta \ph(D) u = f.
		\end{equation}
	\end{theorem}
	\begin{proof}
		Since $r \ge 0$, we have $\delta - r \le \delta$, and by hypothesis we have $\sigma \le s$, so Proposition \ref{derivative_estimates} shows that $-\gamma \partial_1$ and $\ph\pr{D}$ are both bounded linear operators from $X^{s+\sigma}_{r,\delta}(\R^d;\F)$ to $(H^s \cap \dot{H}^{-r})(\R^d;\F)$.  Consider, then, the problem of finding $u$ satisfying \eqref{pde_double_linear_1} for a given $f$.  Applying the Fourier transform, we see that this is equivalent to 
		\begin{equation}
			[-\gamma 2\pi i \xi_1 + \beta \ph(\xi)] \hat{u}(\xi) = \hat{f}(\xi) \text{ for a.e. } \xi \in \R^d.
		\end{equation}
		If a solutions $u$ exists with $f=0$, then since the term in brackets on the left only vanishes at most on a null set, we must have that $\hat{u}=0$ a.e., and hence $u =0$.  Thus, the linear map is injective.  We also learn from this that it is surjective, as we may use this equation to define $\hat{u}$ in terms of $\hat{f}$, and then 
		\begin{multline}
			\norm{u}_{X^{s+\sigma}_{r,\delta}}^2 = \int_{B(0,1)} \frac{\abs{\xi_1}^2 + \abs{\xi}^{2\delta} }{\abs{\xi}^{2r}}  \abs{\hat{u}(\xi)}^2 d\xi 
			+ \int_{B(0,1)^c} \abs{\xi}^{2(s+\sigma)}  \abs{\hat{u}(\xi)}^2 d\xi  \\
			\asymp 
			\int_{B(0,1)} \frac{1}{\abs{\xi}^{2r}}  \abs{\hat{f}(\xi)}^2 d\xi 
			+ \int_{B(0,1)^c} \abs{\xi}^{2s}  \abs{\hat{f}(\xi)}^2 d\xi
			= \norm{f}_{H^s \cap \dot{H}^{-r}}^2,
		\end{multline}
		which shows that $u$ indeed belongs to $X^{s+\sigma}_{r,\delta}(\R^d;\F)$.  Hence, the linear map is an isomorphism.
	\end{proof}

	Next we give an extremely simple but instructive example of how the isomorphism from the previous theorem can be used to solve nonlinear variants of the above traveling wave problem.  Note that the $u$ we obtain from this theorem gives a traveling wave solution by setting $v(x,t) = u(x-\gamma e_1 t)$.
	
	\begin{theorem}\label{pde_double_nonlinear}
        Suppose $\ph: \R^d \to [0, \infty)$ is a continuous function satisfying \eqref{bihom_cond}. 		Let $s, r, \delta, \sigma \in \R$ satisfy $1 < \delta$, $d > 1 + \delta - 2r$, $r \ge 0$, $\sigma \le s$, $1 - r \le s$, and $s + \sigma > d/2$.  If $d=2$, further suppose that $\delta \le 2$.  Suppose that $R>0$ is such that the ball $B(0,R) \subseteq \F$ is the ball of convergence for two analytic functions  $\zeta,\psi : B(0,R) \to \F$  such that $\zeta(0) = \psi(0) =0$ and $\zeta'(0) =\alpha \in \R \backslash \{0\}$ and  $\psi'(0) = \beta \in \R \backslash \{0\}$.  Then there exists an open set $\varnothing \neq \mathcal{U} \subseteq (H^s \cap \dot{H}^{-r})(\R^d;\F)$ such that for each $f \in \mathcal{U}$ there exists a unique $u \in X^{s+\sigma}_{r,\delta}(\R^d;\F)$ satisfying 
		\begin{equation}
			-\gamma \partial_1 [\zeta(u)] +\ph(D) \psi(u) = f.
		\end{equation}
		Moreover, the induced map $\mathcal{U} \ni f \mapsto u \in X^{s+\delta}_{r,\delta}(\R^d;\F)$ is analytic.
	\end{theorem}
	\begin{proof}
		We begin by noting that since $s + \sigma >d/2$, Proposition \ref{thm_fourier_split_estimates} shows that $X^{s+\sigma}_{r,\delta}(\R^d;\F) \hookrightarrow C^0_b(\R^d;\F)$.   Theorem \ref{thm_algebra_result} shows that $X^{s+\sigma}_{r,\delta}(\R^d;\F)$ is an algebra, but it does not show that it is a Banach algebra.  However, by rescaling the norm on $X^{s+\sigma}_{r,\delta}(\R^d;\F)$ by a fixed constant we may assume without loss of generality that $\norm{uv}_{X^{s+\sigma}_{r,\delta}} \le  \norm{u}_{X^{s+\sigma}_{r,\delta}}\norm{v}_{X^{s+\sigma}_{r,\delta}}$.  We may then select an open set $0 \in \mathcal{V} \subseteq X^{s+\sigma}_{r,\delta}(\R^d;\F)$ such that if $u \in \mathcal{V}$ then $u(\R^d) \subseteq B(0,R)$.  Thus, $\zeta\circ u$ and $\psi \circ u$ are well-defined for $u \in \mathcal{V}$, and this induces analytic maps $\zeta,\psi : \mathcal{V}  \to X^{s+\sigma}_{r,\delta}(\R^d;\F)$.
		
		Proposition \ref{derivative_estimates} and the above show that the map $N :   \mathcal{V}  \to  (H^s \cap \dot{H}^{-r})(\R^d;\F)$ defined by $N(u) =-\gamma \partial_1 \zeta(u) +\ph(D) \psi(u)$ is well-defined and analytic, and by construction $N(0) =0$ and its derivatives satisfies $\mathbf{D} N(0) v = -\alpha \gamma \partial_1 v + \beta \varphi(D) v$.  This linear map is an isomorphism thanks to Theorem \ref{pde_double_linear}, and so we may apply the inverse function theorem (see, for instance, Theorem 10.2.5 in \cite{dieudonne2008foundations}) to conclude.
	\end{proof}

	By a similar argument, we can also prove the following variant, which is a nonlinear ``divergence form'' version of the problem from the introduction.
	
	\begin{theorem}\label{pde_single_nonlinear_div_form}
		Let $0 \le s,r \in \R$ and $1 < \delta \in \R$ satisfy $d > 1 + \delta - 2r$ and $s > d/2$.  If $d=2$, further suppose that $\delta \le 2$.  Suppose that $R>0$ is such that the ball $B(0,R) \subseteq \F$ is the ball of convergence for two analytic functions  $\zeta,\psi : B(0,R) \to \F$  such that $\zeta(0) = \psi(0) = 0$,  $\zeta'(0) = \alpha \in \R \backslash \{0\}$, and  $\psi'(0) =\beta \in \R \backslash \{0\}$.  Then there exists an open set $\varnothing \neq \mathcal{U} \subseteq (H^s \cap \dot{H}^{-r})(\R^d;\F)$ such that for each $f \in \mathcal{U}$ there exists a unique $u \in X^{s+\delta}_{r,\delta}(\R^d;\F)$ satisfying 
		\begin{equation}
			-\gamma \partial_1 [\zeta(u)] - (-\Delta)^{\delta/2-1} \diverge  [(1+ \psi(u) )\nabla u]  = f.
		\end{equation}
		Moreover, the induced map $\mathcal{U} \ni f \mapsto u \in X^{s+\delta}_{r,\delta}(\R^d;\F)$ is analytic.
	\end{theorem}

	% _+__+_ -_+__+_ -_+__+_ -_+__+_ -_+__+_ -_+__+_ -_+__+_ -_+__+_ -_+__+_ -_+__+_ -_+__+_ -_+__+_ -_+__+_ -
	\bibliographystyle{abbrv}
	\bibliography{refs.bib}
	% _+__+_ -_+__+_ -_+__+_ -_+__+_ -_+__+_ -_+__+_ -_+__+_ -_+__+_ -_+__+_ -_+__+_ -_+__+_ -_+__+_ -_+__+_ -
	
\end{document}